\documentclass[11pt,a4paper,reqno]{amsart}

\usepackage[T1]{fontenc}
\usepackage[utf8]{inputenc}
\usepackage{lmodern}
\usepackage{microtype}
\usepackage{amsmath,amssymb,amsthm,mathtools}
\usepackage{enumitem}
\usepackage[margin=32mm]{geometry}
\usepackage[colorlinks=true,citecolor=blue,linkcolor=blue,urlcolor=blue]{hyperref}
\hypersetup{
  pdftitle={Cayley-tree pseudo-orbit tracing: period subgroups and relative geometry},
  pdfauthor={Hui Xu}
}

\theoremstyle{plain}
\newtheorem{theorem}{Theorem}[section]
\newtheorem{proposition}[theorem]{Proposition}
\newtheorem{lemma}[theorem]{Lemma}
\newtheorem{corollary}[theorem]{Corollary}
\newtheorem{thma}{Theorem}

\theoremstyle{definition}
\newtheorem{definition}[theorem]{Definition}
\theoremstyle{remark}

\newtheorem{example}[theorem]{Example}
\numberwithin{equation}{section}

\newcommand{\Z}{\mathbb Z}
\newcommand{\N}{\mathbb N}

\newcommand{\Cay}{\operatorname{Cay}}
\newcommand{\diam}{\operatorname{diam}}

\newcommand{\Homeo}{\operatorname{Homeo}}

\title[Cayley-tree POTP]
{Cayley-tree pseudo-orbit tracing: period subgroups and relative geometry}

\author[H.~Xu]{Hui Xu}

\address[H. Xu]{Department of Mathematics, Shanghai Normal University, Shanghai, 200234, China}
\email{huixu@shnu.edu.cn}

\thanks{Hui Xu is supported by NSFC 12201599.}


\begin{document}
\raggedbottom

\begin{abstract}
We introduce Cayley-tree POTP, obtained by imposing the pseudo-orbit
equations of a finitely generated group action only along a spanning tree of
a Cayley graph.  For zero-dimensional actions, we characterize this property
by equicontinuity along normalized replacement paths; for subshifts, the
criterion is expressed in the right-coset space of the common left-period
subgroup.  These criteria characterize virtual freeness and, for
commensurated subgroup pairs, identifies Cayley-tree POTP of the coset full
shift with relative quasi-tree geometry and a finite Bass--Serre
decomposition.  For infinite-index VFP pairs, Cayley-tree POTP of the coset
full shift is equivalent to virtual cohomological codimension one, although
ordinary POTP holds for every such shift.  Finally, we prove that the strong
topological Rokhlin property passes
to finite-index overgroups.  Consequently every finitely generated virtually
free group has this property, answering the virtually cyclic case posed by
Doucha, and Cayley-tree POTP is generic for its Cantor actions.
\end{abstract}

\maketitle

\section{Introduction}

\subsection{Cayley-tree POTP and virtual freeness}

The pseudo-orbit tracing property (POTP), also called shadowing, is a
standard property in hyperbolic dynamics.  It originated in the work
of Anosov and Bowen and has since been studied for actions of finitely
generated groups; see, for example,
\cite{Anosov,BowenAxiomA,Walters,AokiHiraide,Pilyugin,OsipovTikhomirov,ChungLee}.
We use \emph{POTP} consistently for the property and reserve \emph{trace}
for the relation between a point and a pseudo-orbit.

Let a finitely generated group $\Gamma$ act by homeomorphisms on a compact
metric space $(X,d)$, and fix a finite symmetric generating set
$S=S^{-1}$.  Write $C_S=\operatorname{Cay}(\Gamma,S)$ for the left Cayley
graph, with edges $\{g,ag\}$ for $g\in\Gamma$ and $a\in S$.  A family
$(x_g)_{g\in\Gamma}$ is a $\delta$-pseudo-orbit if
\[
 d(ax_g,x_{ag})<\delta\qquad(a\in S,\ g\in\Gamma),
\]
and it is $\varepsilon$-traced by $x\in X$ if
$d(gx,x_g)<\varepsilon$ for every $g\in\Gamma$.  The action has
\emph{ordinary POTP} if, for every $\varepsilon>0$, there exists $\delta>0$
such that every $\delta$-pseudo-orbit is $\varepsilon$-traced.  This notion
is independent of $S$
\cite[Proposition~1]{OsipovTikhomirov}.

We ask how many of the Cayley-edge equations are actually needed.  If
$H\subseteq C_S$ is connected and spanning, an $(H,\delta)$-pseudo-orbit
is required to satisfy $d(ax_g,x_{ag})<\delta$ only when
$\{g,ag\}\in E(H)$.  The action has $H$-POTP if, for every
$\varepsilon>0$, there exists $\delta>0$ such that every
$(H,\delta)$-pseudo-orbit is $\varepsilon$-traced.  For a spanning tree
$T\subseteq C_S$ we write $T$-POTP.  Finally, an action has
\emph{Cayley-tree POTP} if $T$-POTP holds for some finite symmetric
generating set $S$ and some spanning tree $T\subseteq C_S$.  Thus both
$S$ and $T$ may be chosen to suit the action; neither is fixed in advance.
Since deleting edges weakens the pseudo-orbit equations,
\[
 T\text{-POTP}\quad\Longrightarrow\quad\text{ordinary POTP},
\]
and the reverse implication is the substantive issue.
For the full shift, Theorem~\ref{thm:intro-main-universal} shows that this
reverse implication is equivalent to virtual freeness of $\Gamma$.

We first consider symbolic actions.  Throughout,
$\mathcal A$ denotes a finite alphabet.  The full shift
$\mathcal A^\Gamma$ carries the right-shift action
$(\sigma_gx)(h)=x(hg)$.  Formal definitions of subshifts and subshifts of
finite type are given in Subsection~\ref{subsec:prelim-symbolic-periods}.
Chung and Lee proved that a nonempty subshift has
ordinary POTP exactly when it is a subshift of finite type (SFT)
\cite[Theorem~3.2]{ChungLee}.  Hence every full shift has ordinary POTP,
regardless of the geometry of $\Gamma$.  More generally, ordinary POTP holds
for every nonempty SFT.  By contrast, Cayley-tree POTP fails for the full
shift over $\mathbb Z^2$ and, for the coset shifts in
Theorem~\ref{thm:intro-main-commensurated}, is equivalent to relative
quasi-tree geometry.

For a spanning tree $T\subseteq C_S$, let $d_T$ be its path metric and put
\[
 \operatorname{str}_S(T)
 :=\sup\{d_T(g,ag):g\in\Gamma,\ a\in S\}.
\]
Recall that $\Gamma$ is \emph{virtually free} if it contains a free subgroup
of finite index.
\begin{thma}
\label{thm:intro-main-universal}
Let $\Gamma$ be finitely generated and let $|\mathcal A|\geq2$.  The
following are equivalent.
\begin{enumerate}[label=\textup{(\roman*)}]
\item $\Gamma$ is virtually free;
\item some Cayley graph $C_S$ contains a spanning tree $T$ with
$\operatorname{str}_S(T)<\infty$;
\item the full shift $\mathcal A^\Gamma$ has Cayley-tree POTP;
\item every compact metrizable $\Gamma$-action with ordinary POTP has
Cayley-tree POTP.
\end{enumerate}
\end{thma}

The geometric equivalence between virtual freeness and the existence of a
uniform spanning tree is due to Antol\'{\i}n
\cite[Theorem~4.7]{Antolin}.  The dynamical assertions in
Theorem~\ref{thm:intro-main-universal} show that Cayley-tree POTP of the
full shift is equivalent to that geometry, and virtual freeness is exactly the
condition under which \emph{every} ordinary-POTP action can be reduced to a
tree support.  Thus the theorem converts an established geometric
characterization into a dynamical test, rather than adding another spanning
tree criterion.  In particular, the full shift over $\mathbb Z^2$ has
ordinary POTP but no Cayley-tree POTP.  The property also depends on the
chosen tree: Example~\ref{ex:two-spanning-trees} exhibits one action and two
spanning trees in the same Cayley graph for which $T_1$-POTP holds and
$T_2$-POTP fails.

Cayley-tree POTP differs from the POTP of a tree-shift, whose configurations are
indexed by a rooted tree \cite{Bucki}.  It is also different from the
translation-invariant support formalism of \cite{LinChenZhou}: a Cayley
spanning tree may retain different generator edges at different vertices.

\subsection{The replacement-path criterion}

The virtual-freeness characterization follows from an exact criterion for a
fixed tree.  For vertices $u,w$ of $T$, let $[u,w]_T$ denote the unique simple
path, also viewed as its vertex set, and define the normalized replacement
set
\begin{equation}
 \mathcal R_S(T)
 =\{agv^{-1}:a\in S,\ g\in\Gamma,\ v\in[g,ag]_T\}.
 \label{eq:intro-replacement-set}
\end{equation}
Here $agv^{-1}$ is the left multiplier carrying $v$ to the terminal
endpoint $ag$.  The notation $\mathcal R_S(T)$ denotes a subset, not in
general a subgroup.

For a nonempty subshift $X\subseteq\mathcal A^\Gamma$, its common
left-period subgroup is
\[
 P_\ell(X)=\{p\in\Gamma:x(ph)=x(h)
       \text{ for all }x\in X\text{ and }h\in\Gamma\}.
\]
Unlike $\mathcal R_S(T)$, this is always a subgroup, though it need not be
normal.  If $P\leq\Gamma$ and $E\subseteq\Gamma$, then
$P\backslash E:=\{Pg:g\in E\}$ is the image of $E$ in the right-coset
space; no quotient-group structure is intended.  A compact space is
zero-dimensional if it has a basis of clopen sets.  Equicontinuity and the
other dynamical terminology used below are recalled in
Subsection~\ref{subsec:prelim-actions-POTP}.

\begin{thma}
\label{thm:intro-main-replacement}
Let $S=S^{-1}$ be a finite generating set of $\Gamma$ and let
$T\subseteq C_S$ be a spanning tree.
\begin{enumerate}[label=\textup{(\alph*)}]
\item A zero-dimensional compact metrizable $\Gamma$-action has $T$-POTP
if and only if it has ordinary POTP and the family
$\{x\mapsto rx:r\in\mathcal R_S(T)\}$ is equicontinuous.
\item A nonempty subshift $X\subseteq\mathcal A^\Gamma$ has $T$-POTP if
and only if it is an SFT and
\[
 |P_\ell(X)\backslash c\mathcal R_S(T)|<\infty
 \qquad(c\in\Gamma).
\]
Equivalently, $X$ is an SFT and
$\{\sigma_r|_X:r\in\mathcal R_S(T)\}$ is equicontinuous.
\end{enumerate}
\end{thma}

For the full shift, $P_\ell(\mathcal A^\Gamma)=\{e\}$, so the criterion
reduces to bounded stretch.  For a general SFT,
$|P_\ell(X)\backslash c\mathcal R_S(T)|$ may be finite even when
$c\mathcal R_S(T)$ is infinite.  Requiring this finiteness for every
$c\in\Gamma$ depends on how $P_\ell(X)$ intersects its conjugates; we
analyze the finite-height and commensurated cases separately.

\subsection{Coset full shifts}
\label{subsec:intro-two-period-regimes}

We now specialize to shifts indexed by a coset space.  For a subgroup
$P\leq\Gamma$, the \emph{coset full shift}
$\mathcal A^{P\backslash\Gamma}$ is the set of all maps
$P\backslash\Gamma\to\mathcal A$, equipped with the right-shift action
\[
 (\sigma_gx)(Ph)=x(Phg)
 \qquad(g,h\in\Gamma).
\]
Equivalently, it is the subshift of $\mathcal A^\Gamma$ consisting of the
configurations satisfying
\[
 x(pg)=x(g)\qquad(p\in P,\ g\in\Gamma).
\]
When $|\mathcal A|\geq2$, its common left-period subgroup is exactly $P$.
Detailed conventions for
right cosets and the orbital action are given in
Subsections~\ref{subsec:prelim-symbolic-periods} and
\ref{subsec:prelim-coset-graphs}.

The coset criterion in
Theorem~\ref{thm:intro-main-replacement}\textup{(b)} depends on the
intersections of $P$ with its conjugates.  We consider finite height and
commensuration separately.

First suppose that finitely many conjugates of a finitely generated
subgroup $P\leq\Gamma$ have finite total intersection.  Then the finiteness
conditions $|P\backslash c\mathcal R_S(T)|<\infty$ for these conjugates
force $\mathcal R_S(T)$ to be finite.  This is the content of the following
proposition, whose proof is given in
Subsection~\ref{subsec:finite-height-period-subgroups}.

\begin{proposition}
\label{prop:intro-finite-conjugate-core}
Let $\Gamma$ be finitely generated, let $P\leq\Gamma$ be finitely generated,
and suppose that there are $c_1,\ldots,c_m\in\Gamma$ such that
\[
 \bigcap_{i=1}^m c_i^{-1}Pc_i
 \quad\text{is finite}.
\]
If $|\mathcal A|\geq2$, then the coset full shift
$\mathcal A^{P\backslash\Gamma}$ has Cayley-tree POTP if and only if
$\Gamma$ is virtually free.
\end{proposition}

A subgroup $P$ has \emph{finite height} if there is $n$ such that, whenever
$Pg_1,\ldots,Pg_{n+1}$ are distinct right cosets, the intersection
$\bigcap_i g_i^{-1}Pg_i$ is finite.  Hence
Proposition~\ref{prop:intro-finite-conjugate-core} applies, in particular,
to infinite-index finite-height subgroups; see
Subsection~\ref{subsec:finite-height-period-subgroups} for the precise use
of this condition.
Gitik--Mitra--Rips--Sageev proved that quasiconvex subgroups of hyperbolic
groups have finite width, the stronger condition obtained by requiring a
uniform bound for families with pairwise infinite intersections.  Finite
width implies finite height
\cite[Main theorem]{GitikMitraRipsSageev}.  We therefore obtain the following
concrete application.

\begin{corollary}
\label{cor:hyperbolic-quasiconvex-period}
Let $\Gamma$ be word-hyperbolic, let $P\leq\Gamma$ be an infinite-index
quasiconvex subgroup, and let $|\mathcal A|\geq2$.  Then the coset full
shift $\mathcal A^{P\backslash\Gamma}$ has ordinary POTP, and it has
Cayley-tree POTP if and only if $\Gamma$ is virtually free.
\end{corollary}

Thus this coset SFT has ordinary POTP throughout the stated class, while
Cayley-tree POTP holds exactly in the virtually free case.

We next consider commensurated subgroups.  Two subgroups are
\emph{commensurable} when their
intersection has finite index in both, and $P$ is \emph{commensurated} if
$P$ and $gPg^{-1}$ are commensurable for every $g\in\Gamma$.  Its orbital coset
graph $\mathcal O_S(\Gamma,P)$ has vertex set $P\backslash\Gamma$ and
edges $Pg\sim Pag$ for $a\in S$; it is locally finite in this case.  The
reduced Schlichting completion $\Gamma//P$ is the closure of the
permutation image of $\Gamma$ on $P\backslash\Gamma$, in the topology of
pointwise convergence.  This construction is made precise in
Subsection~\ref{subsec:prelim-Schlichting}, in particular in
Theorem~\ref{thm:background-Schlichting}.  A quasi-tree is a graph
quasi-isometric to a tree; the metric conventions are recalled in
Subsection~\ref{subsec:prelim-quasitrees}.

\begin{thma}
\label{thm:intro-main-commensurated}
Let $\Gamma$ be finitely generated, let $P\leq\Gamma$ be finitely generated
and commensurated, and let $|\mathcal A|\geq2$.  The following are
equivalent.
\begin{enumerate}[label=\textup{(\roman*)}]
\item $\mathcal A^{P\backslash\Gamma}$ has Cayley-tree POTP;
\item for one, and hence for every, finite generating set $S$, the orbital
graph $\mathcal O_S(\Gamma,P)$ is a quasi-tree;
\item $\Gamma//P$ acts continuously, properly, and cocompactly on a locally
finite tree;
\item $\Gamma$ is the fundamental group of a finite graph of groups whose
vertex and edge groups, viewed as subgroups of $\Gamma$ through the
Bass--Serre embeddings, are commensurable with $P$.
\end{enumerate}
\end{thma}

For the tree action in \textup{(iii)}, properness means that vertex
stabilizers are compact, while cocompactness means that the quotient graph
is finite.  These notions and their relation to Cayley--Abels graphs and
Bass--Serre theory are explained in
Subsection~\ref{subsec:prelim-Schlichting}.  Conditions \textup{(ii)--(iv)} belong
to the established rough-Cayley-graph and Bass--Serre theory of
commensurated pairs, building on Kr\"on--M\"oller and Carette
\cite{KronMollerTopological,Carette}.  The contribution of
Theorem~\ref{thm:intro-main-commensurated} is the equivalence of the dynamical
condition \textup{(i)} and the geometric condition \textup{(ii)}.  This
equivalence yields a finite relative tree decomposition, rather than only
the existence of one splitting.  Without the commensurated hypothesis, the
quasi-isometry type of one orbital graph does not determine Cayley-tree
POTP: for $P=\mathbb Z^2$ in
$\Gamma=P*\mathbb Z$ it is a tree after loops are removed, but the coset
full shift has no Cayley-tree POTP.

Under the finite-intersection hypothesis,
Theorem~\ref{thm:intro-main-replacement}\textup{(b)} forces bounded
stretch and hence virtual freeness.  Under commensuration, the same criterion
is equivalent to the quasi-tree condition for the orbital graph.

\subsection{Cohomological dimension and generic Cantor actions}

Theorem~\ref{thm:intro-main-commensurated} gives the following
cohomological characterization.
Following Margolis, a group is of type VFP over $\mathbb Z$ if some
finite-index subgroup admits a finite-length resolution of the trivial
module by finitely generated projective modules.  Its virtual
cohomological dimension $\operatorname{vcd}$ is the cohomological dimension
of such a finite-index subgroup; cohomological dimension is the least
possible length of a projective resolution of the trivial module.  These
definitions and the precise results
of Margolis used below are stated in Section~\ref{sec:applications}.

\begin{corollary}
\label{cor:vcd-tree-POTP}
Let $\Gamma$ and $P$ be of type VFP over $\mathbb Z$, with
$P\leq\Gamma$ commensurated and $[\Gamma:P]=\infty$.  For
$|\mathcal A|\geq2$, the shift $\mathcal A^{P\backslash\Gamma}$ has
ordinary POTP, and it has Cayley-tree POTP if and only if
$\operatorname{vcd}(\Gamma)=\operatorname{vcd}(P)+1$.
\end{corollary}

Margolis proved that the equality on the right yields the finite
graph-of-groups decomposition in
Theorem~\ref{thm:intro-main-commensurated}\textup{(iv)}, and characterized the equal-dimension
case by finite index \cite[Theorem~1.2 and Proposition~1.3]{Margolis}.
Corollary~\ref{cor:vcd-tree-POTP} states that this codimension-one equality
is equivalent to Cayley-tree POTP of the coset full shift.  Ordinary POTP
does not make this distinction: it holds for every coset shift in the corollary because $P$ is
finitely generated.  The statement applies to nonnormal commensurated
subgroups and therefore does not reduce to a quotient-group formulation.

The second consequence answers an explicit question in generic Cantor
dynamics.  Let $\mathfrak C$ be the Cantor space and
$\operatorname{Act}_\Gamma(\mathfrak C)$ the space of continuous
$\Gamma$-actions, endowed with the Polish topology of pointwise
convergence in $\operatorname{Homeo}(\mathfrak C)$.  A subset is
\emph{comeager} if its complement is a countable union of nowhere dense
sets.  A countable group has the strong topological Rokhlin property (STRP)
if this action space contains an action with comeager conjugacy class.
Doucha asked whether every virtually cyclic group has STRP and,
more broadly, whether STRP is stable under commensurability
\cite[Section~7, immediately before and in Question~7.2]{DouchaSTRP}.

\begin{corollary}
\label{cor:generic-tree-virtually-free}
Every finitely generated virtually free group has STRP.  For every such
$\Gamma$, Cayley-tree POTP is comeager in
$\operatorname{Act}_\Gamma(\mathfrak C)$.  In particular, every virtually
cyclic group has STRP.
\end{corollary}

The finite-index statement used here is
Theorem~\ref{thm:finite-index-STRP}: STRP passes from a finite-index
subgroup to a finitely generated overgroup.  This theorem proves the
finite-index-overgroup direction of Doucha's permanence question and answers
the explicitly stated virtually cyclic case.  The downward direction needed
for full commensurability invariance remains open.  Doucha's
characterization of STRP by generic ordinary POTP, combined with
Theorem~\ref{thm:intro-main-universal}, then gives the generic
Cayley-tree conclusion.  Thus, for virtually free groups, generic ordinary
POTP is equivalent to generic Cayley-tree POTP.

\subsection*{Organization}
Section~\ref{sec:background} groups the required graph-theoretic,
dynamical, symbolic, and subgroup-pair preliminaries, with the external
results stated before use.  Section~\ref{sec:supports} develops supported
pseudo-orbits and the replacement-path equicontinuity theorem.
Section~\ref{sec:shifts} gives a unified proof of both parts of
Theorem~\ref{thm:intro-main-replacement} and proves
Theorem~\ref{thm:intro-main-universal}.  Section~\ref{sec:main-classification}
provides the proof of Proposition~\ref{prop:intro-finite-conjugate-core},
treats finite-height and commensurated period subgroups, and proves
Theorem~\ref{thm:intro-main-commensurated} and
Corollary~\ref{cor:hyperbolic-quasiconvex-period}.
Section~\ref{sec:applications} proves the cohomological corollary.
Section~\ref{sec:generic-STRP} proves finite-index permanence of STRP and
Corollary~\ref{cor:generic-tree-virtually-free}.

\section{Preliminaries}
\label{sec:background}

Throughout, graphs are undirected and carry their path metrics, $e$ is the identity of
$\Gamma$, and $\mathbb N=\{0,1,2,\ldots\}$.

\subsection{Cayley graphs, spanning trees, and stretch}

 Let $\Gamma$ be a finitely generated group and let $S=S^{-1}$ be a
finite generating set.  We use the left Cayley graph
\[
 C_S=\operatorname{Cay}(\Gamma,S),
 \qquad E(C_S)=\bigl\{\{g,ag\}:g\in\Gamma,\ a\in S\bigr\}.
\]
The choice of a left Cayley graph matches the indices $g$ and $ag$ in the
pseudo-orbit relation.

For a graph $\mathcal G$, write $V(\mathcal G)$ and $E(\mathcal G)$ for its
vertex and edge sets and $d_{\mathcal G}$ for its path metric.  A spanning
tree $T\subseteq\mathcal G$ is a connected acyclic subgraph with
$V(T)=V(\mathcal G)$.  For $u,v\in V(T)$, the notation $[u,v]_T$ denotes the
unique simple path, also viewed as its vertex set.  If
$\{u,v\}\in E(\mathcal G)\setminus E(T)$, then
$[u,v]_T\cup\{\{u,v\}\}$ is the fundamental cycle of the deleted edge, and
$[u,v]_T$ is its \emph{replacement path}.

A spanning tree $T\subseteq\mathcal G$ is a \emph{tree $t$-spanner} if
\[
 d_T(u,v)\leq t\,d_{\mathcal G}(u,v)
 \qquad(u,v\in V(\mathcal G)).
\]
For an unweighted graph it is enough to check this on its edges.  In a
Cayley graph we therefore put
\[
 \operatorname{str}_S(T)
 :=\sup_{a\in S,\,g\in\Gamma}d_T(g,ag).
\]
Thus $T$ is a tree spanner exactly when $\operatorname{str}_S(T)<\infty$;
we also call such a tree \emph{bounded-stretch}.  This is a tree contained in
the prescribed Cayley graph, and the comparison is made by the identity map
on $\Gamma$ \cite{CaiCorneil}.

\subsection{Quasi-trees and tree models}
\label{subsec:prelim-quasitrees}

Quasi-isometry is weaker than bilipschitz equivalence.  Recall that
$f:(Y,d_Y)\to(Z,d_Z)$ is a quasi-isometry if there are $\lambda\geq1$ and
$C\geq0$ such that
\[
 \lambda^{-1}d_Y(y,y')-C
 \leq d_Z(f(y),f(y'))
 \leq \lambda d_Y(y,y')+C
\]
for all $y,y'\in Y$, and every point of $Z$ lies at uniformly bounded
distance from $f(Y)$.  A graph quasi-isometric to a tree is a
\emph{quasi-tree}.  Changing the finite generating set changes a Cayley graph
but not its quasi-isometry class.

An \emph{abstract uniform tree} for a connected graph $\mathcal G$ is a tree
$U$ together with a bijection $V(U)\to V(\mathcal G)$ which is Lipschitz in
both directions.  Its edges need not be edges of $\mathcal G$.  The precise
external result used below is the following.

\begin{theorem}
\cite[Theorem~4.7, \textup{(B1)--(B2)}]{Antolin}
\label{thm:background-Antolin}
For a connected graph $\mathcal G$, the following are equivalent:
\begin{enumerate}[label=(\roman*)]
\item $\mathcal G$ is quasi-isometric to a tree;
\item $\mathcal G$ admits an abstract uniform tree.
\end{enumerate}
If $\mathcal G$ is a Cayley graph of a finitely generated group $\Gamma$,
these conditions are also equivalent to $\Gamma$ being virtually free, that
is, containing a free subgroup of finite index.
\end{theorem}

We also use a second tree model.  A \emph{tree decomposition} of
$\mathcal G$ is a tree $D$ with
bags $(\mathcal B_t)_{t\in V(D)}$, $\mathcal B_t\subseteq V(\mathcal G)$,
such that every edge of $\mathcal G$ has both endpoints in some bag and, for
each $v\in V(\mathcal G)$, the set
$\{t:v\in\mathcal B_t\}$ induces a connected subtree of $D$.  Its
\emph{outer diameter} is
$\sup_t\operatorname{diam}_{\mathcal G}(\mathcal B_t)$.

\begin{theorem}
\cite[Theorem~1.8]{BergerSeymour}
\label{thm:background-Berger-Seymour}
A connected graph is quasi-isometric to a tree if and only if it admits a
tree decomposition of finite outer diameter.
\end{theorem}

Consequently, a bounded-stretch Cayley spanning tree always gives a
quasi-tree, whereas the converse first produces only an abstract tree.  In
the proof of Theorem~\ref{thm:intro-main-universal}, the generating set is
enlarged by a finite ball so that every edge of the abstract tree becomes a
Cayley edge.

\subsection{POTP for group actions}
\label{subsec:prelim-actions-POTP}

Let $\Gamma$ be finitely generated, let $S=S^{-1}$ be a finite generating
set, and let $\Gamma$ act by homeomorphisms on a compact metric space
$(X,d)$.  Write $gx$ for the image of $x$ under $g$.  A
$\delta$-pseudo-orbit with respect to $S$ is a family
$(x_g)_{g\in\Gamma}$ satisfying
\[
 d(ax_g,x_{ag})<\delta
 \qquad(a\in S,\ g\in\Gamma).
\]
It is $\varepsilon$-traced by $x\in X$ if
$d(gx,x_g)<\varepsilon$ for every $g\in\Gamma$.  The action has
\emph{ordinary POTP} if, for every $\varepsilon>0$, there is $\delta>0$ for
which every $\delta$-pseudo-orbit is $\varepsilon$-traced.

On a compact space the property is independent of the compatible metric,
because any two compatible metrics are uniformly equivalent.  The following
external result justifies suppressing the generating set from the name.

\begin{theorem}
\cite[Proposition~1]{OsipovTikhomirov}
\label{thm:background-generating-set-invariance}
If $S_1$ and $S_2$ are finite symmetric generating sets of $\Gamma$, then a
compact $\Gamma$-action has POTP with respect to $S_1$ if and only if it has
POTP with respect to $S_2$.
\end{theorem}

See also \cite[Definitions~2.5--2.7]{ChungLee} for the formulation used
here.

A compact space is \emph{zero-dimensional} if it has a basis of clopen
sets.  A \emph{Cantor space} is a nonempty compact metrizable,
zero-dimensional space without isolated points.  A family
$\mathcal F\subseteq C(X,X)$ is \emph{equicontinuous} if, for every
$\varepsilon>0$, there is $\eta>0$ such that
\[
 d(x,y)<\eta\quad\Longrightarrow\quad
 d(fx,fy)<\varepsilon\quad(f\in\mathcal F).
\]
An action is equicontinuous when its entire family of action maps is
equicontinuous.  Its \emph{action kernel} is the normal subgroup of elements
which act trivially on every point of $X$.

Ordinary POTP uses all edges of $C_S$; supported versions are developed in
Section~\ref{sec:supports}.

\subsection{Subshifts and left-period subgroups}
\label{subsec:prelim-symbolic-periods}

Let $\mathcal A$ be a finite alphabet, always given the discrete topology.
The full shift $\mathcal A^\Gamma=\{x:\Gamma\to\mathcal A\}$ has the product
topology and the right-shift action
\[
 (\sigma_gx)(h)=x(hg)
 \qquad(g,h\in\Gamma,\ x\in\mathcal A^\Gamma).
\]
A \emph{subshift} is a closed shift-invariant subset
$X\subseteq\mathcal A^\Gamma$.  It is a \emph{subshift of finite type}
(SFT) if there are a finite window $D\subseteq\Gamma$ and a set of allowed
patterns $\mathcal W\subseteq\mathcal A^D$ such that
\[
 X=\{x\in\mathcal A^\Gamma:(\sigma_gx)|_D\in\mathcal W
                     \text{ for every }g\in\Gamma\}.
\]
\begin{theorem}
\cite[Theorem~3.2]{ChungLee}
\label{thm:CL-SFT}
Let $\Gamma$ be finitely generated and $\mathcal A$ finite.  A nonempty
subshift $X\subseteq\mathcal A^\Gamma$ has ordinary POTP if and only if it
is of finite type.
\end{theorem}

The full shift has no forbidden patterns and is therefore an SFT.

For a nonempty subshift $X\subseteq\mathcal A^\Gamma$, define its
\emph{common left-period subgroup} by
\begin{equation}
 P_\ell(X)=\{p\in\Gamma:x(ph)=x(h)
       \text{ for every }x\in X\text{ and }h\in\Gamma\}.
\label{eq:left-period-group}
\end{equation}
The adjective ``left'' refers to the multiplication $h\mapsto ph$ in this
identity.  The equivalence classes $P_\ell(X)g$ are right cosets.

\begin{lemma}
\label{lem:left-period-subgroup}
Put $P=P_\ell(X)$ and, for $g\in\Gamma$, let
$\kappa_g:X\to\mathcal A$ be the coordinate observable
$\kappa_g(x)=x(g)$.  Then $P\leq\Gamma$, and
\begin{equation}
 \kappa_g=\kappa_h
 \quad\text{if and only if}\quad Pg=Ph.
\label{eq:coordinate-observable-cosets}
\end{equation}
Moreover,
\begin{equation}
 \ker(\Gamma\curvearrowright X)
 =\operatorname{core}_\Gamma(P)
 :=\bigcap_{h\in\Gamma}h^{-1}Ph.
\label{eq:kernel-period-core}
\end{equation}
Thus $P_\ell(X)$ equals the action kernel if and only if it is normal.
\end{lemma}

\begin{proof}
The identity belongs to $P$.  If $p,q\in P$, then
$x(pqh)=x(qh)=x(h)$ for every $x$ and $h$, so $pq\in P$.  Applying the
period identity for $p$ at $p^{-1}h$ gives
$x(h)=x(p^{-1}h)$, so $p^{-1}\in P$.

If $gh^{-1}\in P$, then $\kappa_g=\kappa_h$.  Conversely, if these
observables agree, applying their equality to $\sigma_{h^{-1}t}x$ gives
$x(gh^{-1}t)=x(t)$ for all $x$ and $t$, hence $gh^{-1}\in P$.  Finally,
$k\in\ker(\Gamma\curvearrowright X)$ precisely when
$\kappa_{hk}=\kappa_h$ for every $h\in\Gamma$.  By
\eqref{eq:coordinate-observable-cosets}, this holds precisely when
$Phk=Ph$ for every $h\in\Gamma$, which is equivalent to
$k\in\bigcap_{h\in\Gamma}h^{-1}Ph$.
\end{proof}

The subgroup $P_\ell(X)$ need not be normal.  For example, let
$\Gamma=S_3$, $P=\langle(12)\rangle$, and
$X=\{0,1\}^{P\backslash S_3}$.  Then $P_\ell(X)=P$, whereas the action
kernel is $\operatorname{core}_{S_3}(P)=\{e\}$.  Thus the distinct coordinate
observables are indexed by the right-coset space
$P_\ell(X)\backslash\Gamma$, which generally has no group structure.

For a subgroup $P\leq\Gamma$, the \emph{coset full shift}
$\mathcal A^{P\backslash\Gamma}$ is the set of maps
$P\backslash\Gamma\to\mathcal A$; equivalently, it is the subshift of
$\mathcal A^\Gamma$ consisting of the configurations satisfying
$x(pg)=x(g)$ for every $p\in P$ and $g\in\Gamma$.

\subsection{Coset spaces and orbital graphs}
\label{subsec:prelim-coset-graphs}

For $P\leq\Gamma$ and $E\subseteq\Gamma$, write
\[
 P\backslash E=\{Pg:g\in E\}\subseteq P\backslash\Gamma.
\]
This denotes the image of $E$ in a right-coset space; it is not a quotient
group and does not assert that $E$ is a subgroup or is $P$-invariant.  Two
subgroups are \emph{commensurable} if their intersection has finite index in
both.  The subgroup $P$ is \emph{commensurated} if $P$ and $gPg^{-1}$ are
commensurable for every $g\in\Gamma$; then $(\Gamma,P)$ is a Hecke pair.

The \emph{orbital coset graph} $\mathcal O_S(\Gamma,P)$ has the
edges $Pg\sim Pag$.  Here $\sim$ means adjacency.  The neighbors of $P$
contributed by $a\in S$ are the right $P$-cosets contained in $PaP$, and
their number is
\[
 [P:P\cap a^{-1}Pa].
\]
Consequently the commensurated hypothesis makes the orbital graph locally
finite.  Different finite generating sets give quasi-isometric orbital
graphs.  This is the graph carrying the right $\Gamma$-action and used in
the coarse classification.

\subsection{Schlichting completions and actions on trees}
\label{subsec:prelim-Schlichting}

Assume now that $P$ is commensurated.  The \emph{reduced Schlichting
completion} $\Gamma//P$ is the closure of the permutation image of $\Gamma$
acting on $P\backslash\Gamma$ by right multiplication, with the topology of
pointwise convergence.  Passing to the permutation image removes the normal
core of $P$.  The following statement combines the cited completion,
uniqueness, and coset-identification results.  The source uses left cosets;
applying the inversion bijection $Pg\mapsto g^{-1}P$ gives the right-coset
formulation below.

\begin{theorem}
\cite[Proposition~3.6, Theorem~3.8, and Proposition~3.9]
{KaliszewskiLandstadQuigg}
\label{thm:background-Schlichting}
Let $P$ be commensurated in $\Gamma$, put $G=\Gamma//P$, and let
$\iota:\Gamma\to G$ be the canonical dense homomorphism.  Then
$\overline P:=\overline{\iota(P)}$ is compact open and
$\iota^{-1}(\overline P)=P$.  The maps
\[
 Pg\longmapsto \overline P\,\iota(g),
 \qquad
 PgP\longmapsto \overline P\,\iota(g)\,\overline P
\]
identify, equivariantly, the discrete and completed coset and double-coset
spaces.  These data characterize the reduced completion up to topological
isomorphism.
\end{theorem}

A \emph{Cayley--Abels graph} of a totally disconnected locally compact group
$G$ is a connected locally finite graph on which $G$ acts continuously and
vertex-transitively with compact open vertex stabilizers.
Theorem~\ref{thm:background-Schlichting} identifies the vertices and shows
that the completion action is
vertex-transitive with compact open stabilizers.  Together with the orbital
edge definition, this makes $\mathcal O_S(\Gamma,P)$ a Cayley--Abels graph
for $\Gamma//P$; local finiteness comes from the Hecke condition, not
directly from Theorem~\ref{thm:background-Schlichting}.

We use the following locally compact Milnor--\v{S}varc lemma.

\begin{theorem}
\cite[Section~3.1]{Carette}
\label{thm:background-lc-Milnor-Svarc}
If a compactly generated locally compact group $G$ acts continuously,
properly, cocompactly, and isometrically on a proper geodesic metric space
$Y$, then $Y$ is quasi-isometric to $G$ with a compact word metric.
\end{theorem}

Consequently every Cayley--Abels graph of $G$ is quasi-isometric to $Y$.
For an action on a locally finite graph, properness is equivalent to compact
vertex stabilizers, and cocompactness to finiteness of the quotient graph.

The non-elementary quasi-tree-to-tree passage used below is the following
result (the theorem cited by Carette is due to Cornulier).  A locally finite
tree is \emph{bushy} if there is $R\geq0$ such that every vertex is within
distance $R$ of a vertex whose removal leaves at least three unbounded
components.  Following Carette, a compactly generated locally compact group
is \emph{non-elementary} if it is neither compact nor two-ended.

\begin{theorem}
\cite[Theorem~3.3]{Carette}
\label{thm:background-lc-bushy-tree}
For a locally compact group $G$, the following are equivalent:
\begin{enumerate}[label=\textup{(\roman*)}]
\item $G$ is non-elementary and acts continuously, properly, and
cocompactly on a locally finite tree;
\item $G$ is the fundamental group of a finite graph of compact groups with
open edge monomorphisms;
\item $G$ is compactly generated and quasi-isometric to a locally finite
bushy tree.
\end{enumerate}
\end{theorem}

The next proposition adds the compact and two-ended cases.

\begin{proposition}
\label{thm:background-lc-quasitree}
Let $G$ be a compactly generated totally disconnected locally compact group.
A Cayley--Abels graph of $G$ is a quasi-tree if and only if $G$ admits a
continuous proper cocompact action on a locally finite tree.
\end{proposition}

\begin{proof}
The converse follows from the locally compact Milnor--\v{S}varc lemma,
Theorem~\ref{thm:background-lc-Milnor-Svarc}.  Suppose that a
Cayley--Abels graph is a quasi-tree.  If it is bounded, then $G$ is compact
and acts properly and cocompactly on a one-vertex tree.  If it is
non-elementary, it is quasi-isometric to a locally finite bushy tree, and
Theorem~\ref{thm:background-lc-bushy-tree} applies.

It remains to consider the two-ended case.  By the characterization of
elementary locally compact groups used in
\cite[Section~3.2]{Carette}, $G$ admits a continuous proper isometric action
on $\mathbb R$.  Its image in $\operatorname{Isom}(\mathbb R)$ is a
noncompact closed totally disconnected subgroup, hence is discrete and
virtually cyclic.  It is therefore cocompact and preserves a simplicial
lattice in $\mathbb R$.  The induced action on the corresponding simplicial
line is continuous, proper, and cocompact.
\end{proof}

Finally, we state the Bass--Serre correspondence in the form used below;
the statement combines the two cited directions.

A \emph{graph of groups} assigns groups to the vertices and edges of a graph,
together with injective homomorphisms from each edge group to the groups at
its incident vertices.  Its fundamental group and associated tree are
understood in the sense of Bass--Serre theory.

\begin{theorem}
\cite[Chapter~I, \S5.3, Theorem~12, and \S5.4, Theorem~13]{SerreTrees}
\label{thm:background-Bass-Serre}
A group acting without inversions on a tree is the fundamental group of the
associated quotient graph of groups, whose local groups are the vertex and
edge stabilizers; conversely, every graph of groups has such a Bass--Serre
tree.
\end{theorem}

It follows directly from the quotient description that a finite graph of
groups gives a cocompact action.  The local valence formula also shows that
if every edge group has finite index in both incident vertex groups, then
the Bass--Serre tree is locally finite.

The later classification passes from the orbital graph to its Schlichting
completion, then to a tree action and a finite graph of groups.

\section{POTP on spanning subgraphs}
\label{sec:supports}

We now define POTP on connected spanning subgraphs of a Cayley graph.  We
compare two supports by the lengths of their replacement paths and then
characterize $T$-POTP for zero-dimensional actions by equicontinuity of the
normalized replacement transformations.

Throughout, $\Gamma$ is a finitely generated group, $S=S^{-1}$ is a finite
generating set, and
\[
 C_S=\Cay(\Gamma,S)
\]
is the Cayley graph with vertex set $\Gamma$ and edges
$\{g,ag\}$, where $a\in S$ and $g\in\Gamma$.  Let
$\alpha:\Gamma\to\Homeo(X)$ be a continuous action on a compact metric space
$(X,d)$.  We write $gx$ for $\alpha_g(x)$.

\begin{definition}
Let $H$ be a subgraph of $C_S$ and let $\delta>0$.  An
$(H,\delta)$-pseudo-orbit is a family $(x_g)_{g\in V(H)}$ in $X$ such that
\[
 d(ax_g,x_{ag})<\delta
\]
whenever $a\in S$ and $\{g,ag\}\in E(H)$.  It is $\varepsilon$-traced by
$x\in X$ if
\[
 d(gx,x_g)<\varepsilon \qquad (g\in V(H)).
\]
The action has $H$-POTP if for every $\varepsilon>0$ there
is $\delta>0$ such that every $(H,\delta)$-pseudo-orbit is
$\varepsilon$-traced.
\end{definition}

When $H=C_S$, this is ordinary POTP.  We shall mainly use connected
spanning supports, so that $V(H)=\Gamma$.

\begin{definition}
The action has \emph{Cayley-tree POTP} if there are a finite symmetric
generating set $S$ and a spanning tree $T\subset C_S$ for which it has
$T$-POTP.
\end{definition}

All these notions are invariant under topological conjugacy.  For a fixed
support $H$, the property is independent of the compatible metric on $X$,
because any two compatible metrics on a compact space are uniformly
equivalent.

We first quantify the propagation of local errors along words of bounded
length.

\begin{lemma}
\label{lem:finite-propagation}
For every $\eta>0$ and $L\in\N$ there exists $\delta>0$ with the following
property.  Suppose $1\leq n\leq L$, $a_1,\ldots,a_n\in S$, and
$y_0,\ldots,y_n\in X$ satisfy
\[
 d(a_i y_{i-1},y_i)<\delta \qquad (1\leq i\leq n).
\]
Then
\[
 d(a_n\cdots a_1y_0,y_n)<\eta.
\]
\end{lemma}

\begin{proof}
This follows by induction on $L$.  At the induction step, use uniform
continuity of the finitely many maps $x\mapsto ax$, $a\in S$, to choose the
error allowed before applying the last generator.  Compactness makes all
choices uniform in the points and in the word of length at most $L$.
\end{proof}

The property is monotone under adding edges: if $H_1$ and $H_2$ have the
same vertex set, $E(H_1)\subseteq E(H_2)$, and the action has $H_1$-POTP,
then it has $H_2$-POTP, because every $(H_2,\delta)$-pseudo-orbit is also an
$(H_1,\delta)$-pseudo-orbit.  Every connected spanning support contains a
spanning tree, and a tree therefore imposes the fewest local equations among
connected spanning supports.

\subsection{Comparison of supports}
\label{sec:comparison}

If every edge of one support can be replaced by a path of uniformly bounded
length in another support, Lemma~\ref{lem:finite-propagation} converts local
errors on the latter support into local errors on the former.

Let $H_1,H_2$ be connected spanning subgraphs of $C_S$.  Define
\[
 L(H_1,H_2)=\sup\{d_{H_1}(u,v):\{u,v\}\in E(H_2)\}.
\]
Thus the identity map from $(\Gamma,d_{H_1})$ to
$(\Gamma,d_{H_2})$ is bilipschitz exactly when both
$L(H_1,H_2)$ and $L(H_2,H_1)$ are finite.

\begin{theorem}
\label{thm:support-comparison}
If $H_1$ and $H_2$ are bilipschitz through the identity map, then an arbitrary
compact $\Gamma$-action has $H_1$-POTP if and only if it has
$H_2$-POTP.
\end{theorem}

\begin{proof}
Suppose first that $L(H_1,H_2)\leq L<\infty$.  Fix an error scale $\eta$.
Choose $\delta$ from Lemma~\ref{lem:finite-propagation} for paths of length at
most $L$.  For every edge $\{u,v\}$ of $H_2$, choose an $H_1$-path from $u$
to $v$ of length at most $L$.  The label product along this path is precisely
$vu^{-1}$.  Hence every $(H_1,\delta)$-pseudo-orbit is an
$(H_2,\eta)$-pseudo-orbit.  It follows that $H_2$-POTP implies
$H_1$-POTP.  Interchanging the two supports proves the converse.
\end{proof}

\begin{definition}
Let $H$ be a connected spanning subgraph of $C_S$.  Its stretch in $C_S$ is
\[
 \operatorname{str}_S(H)
  =L(H,C_S)
  =\sup_{a\in S,\,g\in\Gamma}d_H(g,ag).
\]
If this number is finite, we call $H$ a \emph{bounded-stretch support}.  A
bounded-stretch spanning tree is also called a \emph{tree spanner}.
\end{definition}

Since $H\subseteq C_S$, one always has $d_{C_S}\leq d_H$.  We therefore
obtain the following direct consequence.

\begin{corollary}
\label{cor:ordinary-vs-support}
If $H$ is a bounded-stretch connected spanning support, then, for every
compact $\Gamma$-action, $H$-POTP is equivalent to ordinary POTP.
\end{corollary}

The proof uses bounded stretch to replace each Cayley edge by a path of
uniformly bounded length.  Without such a bound,
Lemma~\ref{lem:finite-propagation} does not provide a single local error
tolerance valid for all replacement paths.

\subsection{Replacement paths and equicontinuity}
\label{subsec:replacement-paths}

Let $S=S^{-1}$ be a finite generating set and let $T\subseteq C_S$ be a
spanning tree.  We write $[u,v]_T$ for the vertex set of the unique
$T$-geodesic from $u$ to $v$.  Define the \emph{normalized replacement-path set}
\begin{equation}
\mathcal R_S(T)=
 \{agv^{-1}:a\in S,\ g\in\Gamma,\ v\in[g,ag]_T\}.
\label{eq:replacement-path-set}
\end{equation}
This is a subset of $\Gamma$, not in general a subgroup.  Taking the two
endpoints $v=ag$ and $v=g$ shows that
\begin{equation}
 \{e\}\cup S\subseteq\mathcal R_S(T),
 \qquad\text{and hence}\qquad
 \langle\mathcal R_S(T)\rangle=\Gamma.
\label{eq:replacement-set-generates}
\end{equation}
Nevertheless, multiplication need not preserve the set.  For example, let
$\Gamma=\langle a\rangle\cong\mathbb Z$, let
$S=\{a,a^{-1}\}$, and take $T=C_S$.  Every replacement path is one Cayley
edge, so
\[
 \mathcal R_S(T)=\{e,a,a^{-1}\},
\]
which is not a subgroup.  All occurrences
of $\mathcal R_S(T)$ below therefore treat it only as an indexing set for a
family of transformations.

The order of the factors in \eqref{eq:replacement-path-set} is forced by
the right-shift convention.  In fact, if
$g=v_0,v_1,\ldots,v_n=ag$ is the $T$-geodesic and
$v_i=b_i v_{i-1}$, then the coordinate transported at the $i$-th step is
$agv_i^{-1}$.

The following theorem applies to arbitrary compact metrizable actions:
$T$-POTP implies equicontinuity of the transformations indexed by all
normalized tree replacement paths.  For subshifts this is equivalent to the
period-coset condition in
Theorem~\ref{thm:intro-main-replacement}\textup{(b)}.

\begin{theorem}
\label{thm:replacement-equicontinuity}
Let a finitely generated group $\Gamma$ act on a compact metrizable space
$X$.  If the action has $T$-POTP for a spanning tree
$T\subseteq C_S$, then the family
\[
 \{x\mapsto rx:r\in\mathcal R_S(T)\}
\]
is equicontinuous.
\end{theorem}

\begin{proof}
Fix a target scale $\eta>0$.  Choose $\varepsilon>0$ so that
$\varepsilon<\eta/3$ and
\[
 d(x,y)<\varepsilon
 \quad\Longrightarrow\quad
 d(ax,ay)<\eta/3\qquad(a\in S).
\]
Choose $\delta>0$ from $T$-POTP at scale $\varepsilon$.
By uniform continuity of the finitely many generator maps, choose
$\theta>0$ so that $d(p,p')<\theta$ implies
\begin{equation}
 d(p,p')<\delta,\qquad
 d(b^{-1}p,b^{-1}p')<\delta,\qquad
 d(ap,ap')<\eta
 \quad(a,b\in S).
\label{eq:replacement-equicont-modulus}
\end{equation}
Take $r\in\mathcal R_S(T)$, say $r=qv^{-1}$, where $q=ag$ and
$v\in[g,q]_T$.  If $v=g$, then $r=a$ and the last inequality in
\eqref{eq:replacement-equicont-modulus} applies.  Suppose $v\ne g$, let $u$ be
the vertex immediately before $v$ on $[g,q]_T$, and write $v=bu$.
Deleting $\{u,v\}$ separates $g$ from $q$.  For $p,p'\in X$ with
$d(p,p')<\theta$, put
\[
 y=v^{-1}p,\qquad y'=v^{-1}p',
\]
and define
\[
 x_h=
 \begin{cases}
  hy,&h\text{ lies in the component containing }g,\\
  hy',&h\text{ lies in the component containing }q.
 \end{cases}
\]
Every tree edge except $\{u,v\}$ has zero error.  Across that edge the two
oriented errors are $d(p,p')$ and
$d(b^{-1}p,b^{-1}p')$, so
\eqref{eq:replacement-equicont-modulus} makes $(x_h)$ a
$(T,\delta)$-pseudo-orbit.  Let $z$ $\varepsilon$-trace it.  Tracing at
$g$ and then applying $a$ gives
\[
 d(qz,qv^{-1}p)<\eta/3,
\]
whereas tracing at $q$ gives
\[
 d(qz,qv^{-1}p')<\varepsilon<\eta/3.
\]
Thus $d(rp,rp')<2\eta/3$.  The number $\theta$ is independent of $r$,
which proves equicontinuity.
\end{proof}

For zero-dimensional compacta, ordinary POTP together with this
equicontinuity condition also implies $T$-POTP.  After establishing the
symbolic description of equicontinuity, we give a single proof of both parts of
Theorem~\ref{thm:intro-main-replacement} in
Subsection~\ref{subsec:proof-main-replacement}.

\section{Symbolic criteria}
\label{sec:shifts}

Every nonempty SFT has ordinary POTP.  For a fixed spanning tree $T$,
Theorem~\ref{thm:intro-main-replacement}\textup{(b)} shows that $T$-POTP
depends on the cardinalities of
$P_\ell(X)\backslash c\mathcal R_S(T)$.  For the full shift
$P_\ell(X)=\{e\}$, so this condition is precisely bounded stretch; a
nontrivial period subgroup may make these coset images finite even when the
replacement paths have unbounded length.

\subsection{Full shifts and bounded stretch}

Throughout this section, $S=S^{-1}$ denotes a fixed finite generating set
of $\Gamma$.  We retain the right-shift and finite-type conventions of
Section~\ref{sec:background}.

We use the SFT--POTP theorem stated as
Theorem~\ref{thm:CL-SFT} in the preliminaries.

For any connected spanning support $H$, $H$-POTP implies ordinary POTP and
hence forces a nonempty subshift to be of finite type by
Theorem~\ref{thm:CL-SFT}.  If $H$ has bounded stretch, the converse follows
from Corollary~\ref{cor:ordinary-vs-support}.  The full shift shows that the
bounded-stretch condition is necessary as well as sufficient.

\begin{theorem}
\label{thm:full-shift-detector}
Let $\mathcal A$ be a finite set with at least two elements and let $H$ be a connected
spanning subgraph of $C_S$.  The full shift $\mathcal A^\Gamma$ has $H$-POTP if
and only if
\[
 \operatorname{str}_S(H)<\infty.
\]
\end{theorem}

\begin{proof}
If the stretch is finite, the full shift has ordinary POTP by
Theorem~\ref{thm:CL-SFT}, and
Corollary~\ref{cor:ordinary-vs-support} gives $H$-POTP.

For the converse, it is enough to use two symbols, denoted by $0$ and $1$.
Fix a compatible metric on $\mathcal A^\Gamma$.  Choose $\varepsilon>0$ so small that
\begin{equation}
 d(z,z')<\varepsilon
 \quad\Longrightarrow\quad
 z|_{\{e\}\cup S}=z'|_{\{e\}\cup S}.
\label{eq:epsilon-cylinder}
\end{equation}
Choose $\delta>0$ from $H$-POTP for this $\varepsilon$.  There
is a finite set $D\subset\Gamma$ such that
\begin{equation}
 z|_D=z'|_D \quad\Longrightarrow\quad d(z,z')<\delta.
\label{eq:delta-cylinder}
\end{equation}
Enlarge $D$, if necessary, so that $e\in D$, and put
\[
 D_S=D\cup DS.
\]
Suppose, toward a contradiction, that some $a\in S$ and $g\in\Gamma$ satisfy
\begin{equation}
 d_H(g,ag)>|D_S|.
\label{eq:long-H-edge}
\end{equation}
Put $q=ag$.  On the set of all coordinate variables
$\Gamma\times\Gamma$, introduce the equivalence relation generated by
\begin{equation}
 (v,kb)\sim(bv,k)
 \quad\text{whenever }\{v,bv\}\in E(H),\ b\in S,\ k\in D.
\label{eq:coordinate-relation}
\end{equation}
The product of the two entries is invariant under every generating relation:
\[
 (kb)v=k(bv).
\]
Consider the two variables $(g,a)$ and $(q,e)$, both of which have product
$q$.  If they were equivalent, projecting an equivalence chain to its first
entry would give an $H$-walk from $g$ to $q$.  Every vertex $v$ occurring in
this projected walk satisfies
\[
 qv^{-1}\in D_S.
\]
Indeed, the two relative coordinates at a generating step are $k$ and $kb$
with $k\in D$ and $b\in S$.  Since the map $v\mapsto qv^{-1}$ is injective,
deleting repetitions from the projected walk would produce an $H$-path from
$g$ to $q$ of length at most $|D_S|-1$.  This contradicts
\eqref{eq:long-H-edge}.  Hence $(g,a)$ and $(q,e)$ belong to different
equivalence classes.

Assign the symbol $1$ to the equivalence class of $(q,e)$ and the symbol $0$
to every other class.  For each $v\in\Gamma$, define a configuration
$x_v\in\{0,1\}^\Gamma$ by taking the value of $x_v(h)$ to be the symbol
assigned to the class of $(v,h)$.  Relation
\eqref{eq:coordinate-relation} says exactly that
\[
 (\sigma_bx_v)|_D=x_{bv}|_D
\]
on every edge $\{v,bv\}$ of $H$.  By \eqref{eq:delta-cylinder}, the family
$(x_v)_{v\in\Gamma}$ is an $(H,\delta)$-pseudo-orbit.

If a point $z$ $\varepsilon$-traced it, then the comparison at the vertex
$g$, in coordinate $a$, would give $z(q)=x_g(a)=0$, whereas the comparison
at the vertex $q$, in coordinate $e$, would give $z(q)=x_q(e)=1$.  This
contradicts \eqref{eq:epsilon-cylinder}.  Therefore
$d_H(g,ag)\leq |D_S|$ for all $a\in S$ and $g\in\Gamma$, which proves finite
stretch.
\end{proof}

The following examples show, respectively, that ordinary POTP need not imply
Cayley-tree POTP and that $T$-POTP can hold for one spanning tree and fail
for another spanning tree in the same Cayley graph.

\begin{example}[Ordinary POTP without Cayley-tree POTP]
\label{ex:Z2-full-shift-no-tree}
Let $\Gamma=\mathbb Z^2$ and let $\mathcal A$ be a finite alphabet with
$|\mathcal A|\geq2$.  The full shift $\mathcal A^{\mathbb Z^2}$ is an SFT and hence has
ordinary POTP by Theorem~\ref{thm:CL-SFT}.  If it had $T$-POTP for a
spanning tree of some Cayley graph, then
Theorem~\ref{thm:full-shift-detector} would make that tree a bounded-stretch
spanning tree.  The Cayley graph would consequently be quasi-isometric to a
tree, forcing $\mathbb Z^2$ to be virtually free by
Theorem~\ref{thm:background-Antolin}, which it is not.  Thus this classical full shift
has ordinary POTP but fails $T$-POTP for every spanning tree of every Cayley
graph of $\mathbb Z^2$.
\end{example}

\begin{example}[Dependence on the spanning tree]
\label{ex:two-spanning-trees}
Let
\[
 \Gamma=\mathbb Z\times C_2
 =\langle a,b\mid b^2=e,\ ab=ba\rangle,
 \qquad S=\{a,a^{-1},b\},
\]
and let $X=\{0,1\}^{\Gamma}$ be the full shift.  Its Cayley graph is the
infinite ladder with vertices $a^n$ and $a^nb$, $n\in\mathbb Z$.  Define
two spanning subgraphs by
\[
 \begin{aligned}
 E(T_1)={}&
 \bigl\{\{a^n,a^{n+1}\}:n\in\mathbb Z\bigr\}
 \cup
 \bigl\{\{a^n,a^nb\}:n\in\mathbb Z\bigr\},\\
 E(T_2)={}&
 \bigl\{\{a^n,a^{n+1}\}:n\in\mathbb Z\bigr\}
 \cup
 \bigl\{\{a^nb,a^{n+1}b\}:n\in\mathbb Z\bigr\}
 \cup\bigl\{\{e,b\}\bigr\}.
 \end{aligned}
\]
The graph $T_1$ is one rail together with all rungs, while $T_2$ consists
of both rails joined by the single rung at the identity.  Hence both are
connected, acyclic, and spanning.  Every omitted rail edge has a
$T_1$-replacement path of length three, so
$\operatorname{str}_S(T_1)\leq3$.  On the other hand,
\[
 d_{T_2}(a^n,a^nb)=2|n|+1\longrightarrow\infty,
\]
although $\{a^n,a^nb\}$ is a Cayley edge.  Therefore
Theorem~\ref{thm:full-shift-detector} gives
\[
 X\text{ has }T_1\text{-POTP},
 \qquad
 X\text{ does not have }T_2\text{-POTP}.
\]
Thus $T$-POTP depends on both the action and the chosen spanning tree,
whereas Cayley-tree POTP requires the existence of at least one such tree.
\end{example}

\subsection{Virtual freeness}

An abstract uniform tree becomes a spanning-tree subgraph after the
generating set is enlarged by a finite ball.

\begin{proof}[Proof of Theorem~\ref{thm:intro-main-universal}]
The equivalence of \textup{(ii)} and \textup{(iii)} is
Theorem~\ref{thm:full-shift-detector}.  If \textup{(ii)} holds, the identity
map between $C_S$ and $T$ is bilipschitz.  Hence $\Gamma$ is quasi-isometric
to a tree and is virtually free by
Theorem~\ref{thm:background-Antolin}.

Conversely, suppose that $\Gamma$ is virtually free and begin with any
finite symmetric generating set $S_0$.  By
Theorem~\ref{thm:background-Antolin}, the graph $C_{S_0}$ admits an abstract
uniform tree: there are a tree $U$ and a bijection
\[
 \phi:V(U)\longrightarrow\Gamma
\]
such that both $\phi$ and $\phi^{-1}$ are Lipschitz.  Transport the tree
structure through $\phi$ and regard $U$ as a tree on $\Gamma$.  There is an
integer $L\geq1$ such that the endpoints of every $U$-edge have
$S_0$-distance at most $L$.  Put
\[
 S=B_{S_0}(e,L)\setminus\{e\}.
\]
Then $S$ is finite and symmetric, contains $S_0$, and every edge of $U$ is
an edge of $C_S$.  Thus $U\subseteq C_S$ is a spanning tree.

Let $M$ be a Lipschitz constant for $\phi^{-1}$.  If $\{g,ag\}$ is an edge
of $C_S$, then $d_{S_0}(g,ag)\leq L$, and therefore
\[
 d_U(g,ag)\leq M d_{S_0}(g,ag)\leq ML.
\]
Hence $\operatorname{str}_S(U)\leq ML$, proving \textup{(ii)}.  Condition
\textup{(ii)} implies \textup{(iv)} by
Corollary~\ref{cor:ordinary-vs-support}.  Finally, \textup{(iv)} implies
\textup{(iii)} because the full shift $\mathcal A^\Gamma$ is an SFT and
therefore has ordinary POTP.
\end{proof}

\subsection{A criterion for a fixed SFT}
\label{sec:individual-SFT}

For a fixed spanning tree, we characterize the subshifts with $T$-POTP in
terms of their common left-period subgroup, which need not equal the action
kernel.  The proof expresses the required equicontinuity through finite
determinacy sets.

Let $\mathcal A$ be a finite alphabet and let
$X\subseteq\mathcal A^\Gamma$ be a nonempty subshift.  We use the common
left-period subgroup, its coordinate-coset description, and its relation to
the action kernel from Lemma~\ref{lem:left-period-subgroup}.

For a finite set $D\subseteq\Gamma$, put
\begin{equation}
 \operatorname{Det}_X(D)=
 \{t\in\Gamma:x|_D=y|_D\Longrightarrow x(t)=y(t)
                    \text{ for all }x,y\in X\}.
\label{eq:determinacy-hull}
\end{equation}
This is the set of coordinates whose symbols are determined by the pattern
on $D$.  Like $\mathcal R_S(T)$, it need not be a subgroup; the next lemma
shows instead that it is saturated by right $P_\ell(X)$-cosets.

\begin{lemma}
\label{lem:finite-determinacy-periods}
Let $X\subseteq \mathcal A^\Gamma$ be a nonempty subshift and let
$P=P_\ell(X)$.  For every finite $D\subseteq\Gamma$,
\begin{equation}
 P\operatorname{Det}_X(D)=\operatorname{Det}_X(D),
 \qquad
 |P\backslash\operatorname{Det}_X(D)|<\infty,
 \quad\text{and}\quad
 PD\subseteq\operatorname{Det}_X(D).
\label{eq:det-cosets}
\end{equation}
Consequently, for an arbitrary set $E\subseteq\Gamma$, the following are
equivalent:

\begin{enumerate}[label=(\roman*)]
\item $E\subseteq\operatorname{Det}_X(D)$ for some finite
      $D\subseteq\Gamma$;
\item $|P\backslash E|<\infty$.
\end{enumerate}
\end{lemma}

\begin{proof}
By Lemma~\ref{lem:left-period-subgroup}, the coordinate observables satisfy
$\kappa_t=\kappa_s$ if and only if $ts^{-1}\in P$.  If
$t\in\operatorname{Det}_X(D)$, then $\kappa_t$ factors through the
finite set of admissible $D$-patterns.  There are only finitely many maps
from that finite set to $\mathcal A$.  Hence only finitely many distinct coordinate
maps occur among the $\kappa_t$ with
$t\in\operatorname{Det}_X(D)$.  By
\eqref{eq:coordinate-observable-cosets},
each equality class is one right $P$-coset.  This proves the finiteness
assertion in \eqref{eq:det-cosets}.

If $p\in P$ and $t\in\operatorname{Det}_X(D)$, then
$\kappa_{pt}=\kappa_t$, so $pt\in\operatorname{Det}_X(D)$.  Hence
$P\operatorname{Det}_X(D)\subseteq\operatorname{Det}_X(D)$; the reverse
inclusion follows from $e\in P$.  Finally,
$D\subseteq\operatorname{Det}_X(D)$, and therefore
$PD\subseteq\operatorname{Det}_X(D)$.

If (i) holds, the first part of \eqref{eq:det-cosets} gives (ii).  If (ii)
holds, choose a finite set $D$ containing one representative of every right
$P$-coset which meets $E$.  Then
$E\subseteq PD\subseteq\operatorname{Det}_X(D)$.
\end{proof}

For example, for the full shift $\mathcal A^{\mathbb Z}$ with $|\mathcal A|\geq2$, one has
$P_\ell(X)=\{0\}$ and $\operatorname{Det}_X(D)=D$.  Taking
$D=\{0,1\}$ shows explicitly that $\operatorname{Det}_X(D)$ need not be a
subgroup.

Lemma~\ref{lem:finite-determinacy-periods} has the following dynamical
reformulation.  Recall that a
family of maps on a compact space is \emph{equicontinuous} when one modulus
of continuity works for every map in the family.

\begin{proposition}
\label{prop:period-equicontinuity}
Let $X\subseteq \mathcal A^\Gamma$ be a nonempty subshift, let
$P=P_\ell(X)$, and let $E\subseteq\Gamma$.  Then the family of restricted
shifts
\[
 \{\sigma_r|_X:r\in E\}
\]
is equicontinuous if and only if
\[
 |P\backslash cE|<\infty
 \qquad\text{for every }c\in\Gamma.
\]
\end{proposition}

\begin{proof}
In the product uniformity, equicontinuity says that for every finite target
window $C\subseteq\Gamma$ there is a finite source window
$D\subseteq\Gamma$ such that
\[
 x|_D=y|_D
 \quad\Longrightarrow\quad
 (\sigma_rx)|_C=(\sigma_ry)|_C
 \qquad(r\in E).
\]
Since $(\sigma_rx)(c)=x(cr)$, this is exactly the inclusion
$CE\subseteq\operatorname{Det}_X(D)$.  By
Lemma~\ref{lem:finite-determinacy-periods}, such a $D$ exists for every
finite $C$ exactly when $P\backslash cE$ is finite for every $c$.
\end{proof}

\subsection{Proof of Theorem~\ref{thm:intro-main-replacement}}
\label{subsec:proof-main-replacement}

We now combine Theorem~\ref{thm:replacement-equicontinuity},
Theorem~\ref{thm:CL-SFT}, and
Proposition~\ref{prop:period-equicontinuity} to prove the replacement-path
criterion.

\begin{proof}[Proof of Theorem~\ref{thm:intro-main-replacement}]
For part~\textup{(a)}, suppose first that the action has $T$-POTP.  It has
ordinary POTP by monotonicity, and its replacement family is equicontinuous
by Theorem~\ref{thm:replacement-equicontinuity}.

Conversely, suppose that the action has ordinary POTP and that the
replacement family is equicontinuous.  Fix $\varepsilon>0$, and choose
$\eta>0$ so that every ordinary $\eta$-pseudo-orbit is
$\varepsilon$-traced.  Choose a finite clopen partition $\mathcal P$ of
$X$ whose atoms have diameter less than $\eta$.

If $\mathcal P$ has only one atom, then $\operatorname{diam}(X)<\eta$, so
every family indexed by $\Gamma$ is already an ordinary
$\eta$-pseudo-orbit and is therefore $\varepsilon$-traced.  We may thus
assume that $\mathcal P$ has at least two atoms and put
\[
 \lambda=\min\{d(A,B):A,B\in\mathcal P,\ A\ne B\}>0.
\]
Equicontinuity gives $\delta>0$ such that
\begin{equation}
 d(x,y)<\delta
 \quad\Longrightarrow\quad
 d(rx,ry)<\lambda
 \qquad(r\in\mathcal R_S(T)).
 \label{eq:replacement-clopen-modulus}
\end{equation}
Thus the two images in \eqref{eq:replacement-clopen-modulus} belong to the
same atom of $\mathcal P$.

Let $(x_h)_{h\in\Gamma}$ be a $(T,\delta)$-pseudo-orbit.  Fix $a\in S$ and
$g\in\Gamma$, put $q=ag$, and write
\[
 g=v_0,v_1,\ldots,v_n=q,
 \qquad v_i=b_i v_{i-1}\quad(b_i\in S)
\]
for the $T$-geodesic.  Set
\[
 r_i=qv_i^{-1},\qquad y_i=r_i x_{v_i}\qquad(0\leq i\leq n).
\]
Each $r_i$ belongs to $\mathcal R_S(T)$ and
$r_{i-1}=r_i b_i$.  The tree pseudo-orbit inequality and
\eqref{eq:replacement-clopen-modulus} imply that
\[
 y_{i-1}=r_i b_i x_{v_{i-1}}
 \quad\text{and}\quad
 y_i=r_i x_{v_i}
\]
belong to the same atom of $\mathcal P$.  Hence all
$y_0,\ldots,y_n$ lie in one atom.  Since
\[
 y_0=ax_g,\qquad y_n=x_{ag},
\]
we obtain $d(ax_g,x_{ag})<\eta$.  The family is therefore an ordinary
$\eta$-pseudo-orbit and is $\varepsilon$-traced.  This proves part~\textup{(a)}.

For part~\textup{(b)}, let $X\subseteq\mathcal A^\Gamma$ be nonempty and
put $P=P_\ell(X)$.  By Theorem~\ref{thm:CL-SFT}, ordinary POTP is equivalent
to $X$ being of finite type.  By part~\textup{(a)}, $T$-POTP is equivalent
to ordinary POTP together with equicontinuity of
$\{\sigma_r|_X:r\in\mathcal R_S(T)\}$.  Applying
Proposition~\ref{prop:period-equicontinuity} with
$E=\mathcal R_S(T)$ shows that this equicontinuity is equivalent to
\[
 \bigl|P\backslash c\mathcal R_S(T)\bigr|<\infty
 \qquad(c\in\Gamma).
\]
This proves both formulations in part~\textup{(b)}.
\end{proof}

Lemma~\ref{lem:finite-determinacy-periods} gives a finite-coordinate form of
Theorem~\ref{thm:intro-main-replacement}\textup{(b)}: for a nonempty SFT
$X$ with $P=P_\ell(X)$, the action has $T$-POTP if and only if, for every
finite $C\subseteq\Gamma$, there is a finite $D\subseteq\Gamma$ such that
\[
 C\mathcal R_S(T)\subseteq\operatorname{Det}_X(D).
\]
In particular, among nonempty SFTs the trees supporting POTP depend only on
the common left-period subgroup.

\subsection{An SFT with unbounded stretch}

If $P_\ell(X)$ is finite, the criterion reduces to bounded stretch: the
finite-coset condition is equivalent to finiteness of
$\mathcal R_S(T)$, and the latter is equivalent to a uniform bound on the
replacement-path lengths.  Consequently, over a non-virtually-free group,
every SFT with Cayley-tree POTP has an infinite common left-period subgroup.
The following example shows that the condition
$|P_\ell(X)|=\infty$ does not imply that $X$ is finite.

\begin{example}[An infinite $\mathbb Z^2$ SFT with $T$-POTP]
\label{ex:comb-quotient-SFT}
Let $\Gamma=\mathbb Z^2$, let
$S=\{\pm e_1,\pm e_2\}$, and let
\begin{equation}
 X=\{x\in\{0,1\}^{\mathbb Z^2}:
       x(m,n)=x(m,n+1)\text{ for all }(m,n)\in\mathbb Z^2\}.
\label{eq:vertical-constant-SFT}
\end{equation}
This is an infinite SFT, conjugate to the full two-shift over $\mathbb Z$;
the second generator acts trivially.  Let $T$ be the comb consisting of all
vertical edges together with the horizontal edges on
$\mathbb Z\times\{0\}$.  The tree $T$ has unbounded stretch: the
$T$-distance between $(m,n)$ and $(m+1,n)$ is $2|n|+1$.

Here $P_\ell(X)=\mathbb Ze_2$.  If $q$ and $g$ are Cayley neighbors, the
projection to the first coordinate of every vertex on $[g,q]_T$ differs
from that of $q$ by at most one.  Therefore
\[
 \mathcal R_S(T)\subseteq
 P_\ell(X)\{e,e_1,e_1^{-1}\}.
\]
Since $P_\ell(X)$ is normal, the same finite-coset conclusion holds for
every left translate $c\mathcal R_S(T)$.  By
Theorem~\ref{thm:intro-main-replacement}\textup{(b)}, $X$ has $T$-POTP.
Thus failure of Cayley-tree POTP for the full shift over $\mathbb Z^2$ does
not imply failure for every infinite $\mathbb Z^2$ SFT.
\end{example}

\subsection{Coset full shifts}

We next apply the period criterion to coset full shifts.  Let
$P\leq\Gamma$ and let $\mathcal A$ be a finite alphabet with at least two symbols.
Denote by
$\mathcal A^{P\backslash\Gamma}$ the subshift of configurations constant on the right
$P$-cosets.  No normality assumption is needed: right translations preserve
the partition into right $P$-cosets.

\begin{corollary}
\label{cor:coset-full-shift-tree}
Let $S=S^{-1}$ be a finite generating set of $\Gamma$, let $P\leq\Gamma$,
and let $\mathcal A$ be a finite alphabet with $|\mathcal A|\geq2$.  For a spanning tree
$T\subseteq C_S$,
the coset full shift $\mathcal A^{P\backslash\Gamma}$ has $T$-POTP if and only
if

\begin{enumerate}[label=(\roman*)]
\item $P$ is finitely generated as a subgroup of $\Gamma$; and
\item $|P\backslash c\mathcal R_S(T)|<\infty$ for every $c\in\Gamma$.
\end{enumerate}

If $P$ is commensurated in $\Gamma$, condition (ii) is equivalent to
$|P\backslash\mathcal R_S(T)|<\infty$.  In particular, if
$P\lhd\Gamma$, it is equivalent to finiteness of the image of
$\mathcal R_S(T)$ in $\Gamma/P$.
\end{corollary}

\begin{proof}
The left period subgroup of the coset full shift is exactly $P$.  Moreover,
this subshift is of finite type exactly when $P$ is finitely generated.  If
$P$ is generated by a finite set $B$, the relations
$x(bg)=x(g)$, $b\in B$, are finite defining rules.  Conversely, suppose a
finite window $F$ defines the coset full shift.  Let $L$ be the subgroup of
$P$ generated by the finite set
\[
 \{f_1f_2^{-1}:f_1,f_2\in F,\ f_1f_2^{-1}\in P\}.
\]
Every configuration constant on right $L$-cosets has, on each translate of
$F$, a pattern which extends to a configuration constant on right
$P$-cosets.  It therefore satisfies all the defining rules.  Equality with
the coset full shift forces $L=P$, so $P$ is finitely generated.  The first
assertion now follows from
Theorem~\ref{thm:intro-main-replacement}\textup{(b)}.

If $P$ is commensurated, every double coset $PcP$ is a finite union of right
$P$-cosets.  It follows that finiteness of
$P\backslash\mathcal R_S(T)$ implies finiteness of
$P\backslash c\mathcal R_S(T)$ for every $c$; the reverse implication is
the case $c=e$.  The final assertion follows by normality.
\end{proof}

\section{Coset shifts and period subgroups}
\label{sec:main-classification}

The translate conditions in
Corollary~\ref{cor:coset-full-shift-tree} have different consequences in two
cases.  If finitely many conjugates of $P$ have finite intersection, they
force $\mathcal R_S(T)$ to be finite and hence force virtual freeness of
$\Gamma$.  If $P$ is commensurated, they are equivalent to a quasi-tree
condition on the orbital graph.

\subsection{Finite intersections of conjugates}
\label{subsec:finite-height-period-subgroups}

\begin{proof}[Proof of Proposition~\ref{prop:intro-finite-conjugate-core}]
Choose $c_1,\ldots,c_m$ as in the proposition and put
\begin{equation}
 K:=\bigcap_{i=1}^m c_i^{-1}Pc_i.
 \label{eq:finite-conjugate-core}
\end{equation}
This group is finite by hypothesis.
If $\Gamma$ is virtually free, then
$\mathcal A^{P\backslash\Gamma}$ is an SFT because $P$ is finitely
generated, and hence it has ordinary POTP by
Theorem~\ref{thm:CL-SFT}.  Theorem~\ref{thm:intro-main-universal} implies
Cayley-tree POTP.

Conversely, choose a finite symmetric generating set $S$ and a spanning
tree $T\subseteq C_S$ witnessing Cayley-tree POTP, and put
$E=\mathcal R_S(T)$.  By
Corollary~\ref{cor:coset-full-shift-tree},
\[
 |P\backslash c_iE|<\infty
 \qquad(1\leq i\leq m).
\]
Left multiplication by $c_i^{-1}$ induces a bijection
\[
 P\backslash c_iE\longrightarrow c_i^{-1}Pc_i\backslash E,
 \qquad
 Pc_ir\longmapsto c_i^{-1}Pc_i r.
\]
Thus, for every $i$, the set $E$ is contained in finitely many right cosets
of $c_i^{-1}Pc_i$.  Distributing the intersection of these finite unions
expresses $E$ as a subset of finitely many sets of the form
\[
 \bigcap_{i=1}^m(c_i^{-1}Pc_i)r_i.
\]
If such an intersection contains $x$, then it equals $Kx$, and is therefore
finite by \eqref{eq:finite-conjugate-core}.  Hence $E$ is finite.

For a fixed replacement path with terminal endpoint $q$, the map
$v\mapsto qv^{-1}$ is injective.  Thus finiteness of $E$ bounds the lengths
of all replacement paths uniformly, so $T$ has bounded stretch.
Theorem~\ref{thm:intro-main-universal} implies that $\Gamma$ is virtually free.
\end{proof}

A subgroup $P$ has \emph{finite height} if there is $n$ such that
$\bigcap_{i=1}^{n+1}g_i^{-1}Pg_i$ is finite whenever
$Pg_1,\ldots,Pg_{n+1}$ are distinct right cosets.
Consequently, Proposition~\ref{prop:intro-finite-conjugate-core} applies to every
infinite-index finite-height subgroup.  We use the following classical
source of finite-height subgroups.

The subgroup $P$ has \emph{finite width} if there is $n$ such that, whenever
$Pg_1,\ldots,Pg_{n+1}$ are distinct right cosets, some pair of conjugates
$g_i^{-1}Pg_i$ and $g_j^{-1}Pg_j$ has finite intersection.

\begin{theorem}
\cite[Main theorem]{GitikMitraRipsSageev}
\label{thm:background-quasiconvex-finite-height}
Every quasiconvex subgroup of a word-hyperbolic group has finite width.
\end{theorem}

Since an infinite total intersection forces all pairwise intersections to
be infinite, finite width implies finite height.

\begin{proof}[Proof of Corollary~\ref{cor:hyperbolic-quasiconvex-period}]
The subgroup $P$ is finitely generated.  By
Theorem~\ref{thm:background-quasiconvex-finite-height} and the implication
from finite width to finite height, $P$ has finite height.
Let $n$ be as in the definition of finite height.  Since
$[\Gamma:P]=\infty$, choose distinct right cosets
$Pg_1,\ldots,Pg_{n+1}$; then
$\bigcap_{i=1}^{n+1}g_i^{-1}Pg_i$ is finite.
Proposition~\ref{prop:intro-finite-conjugate-core} gives the asserted if-and-only-if
statement.  Ordinary POTP follows
because $P$ is finitely generated, so the coset full shift is an SFT.
\end{proof}

Without commensuration, the quasi-isometry type of the orbital graph alone
does not determine Cayley-tree POTP.  Let $P=\mathbb Z^2$ and
$\Gamma=P*\langle t\rangle$, and take the standard generators of $P$
together with $t^{\pm1}$.  The free factor $P$ is malnormal, while, after
discarding loops, $\mathcal O_S(\Gamma,P)$ is an infinite-valence tree: a
non-backtracking closed path would contradict the free-product normal form.
Nevertheless, $\Gamma$ is not virtually free, since it contains
$\mathbb Z^2$.  Proposition~\ref{prop:intro-finite-conjugate-core} therefore shows that
the coset full shift has ordinary POTP but not Cayley-tree POTP.  Thus the
orbital graph does not by itself determine whether all quantities
$|P\backslash c\mathcal R_S(T)|$ in
Theorem~\ref{thm:intro-main-replacement}\textup{(b)} are finite.

\subsection{Commensurated subgroups}

We prove Theorem~\ref{thm:intro-main-commensurated}.  For a commensurated
subgroup, the next lemma converts an abstract uniform tree on the orbital
coset space into a Cayley spanning tree satisfying the finite-coset condition
of Corollary~\ref{cor:coset-full-shift-tree}.

\begin{lemma}
\label{lem:uniform-coset-tree-lifting}
Let $\Gamma$ be finitely generated and let $P\leq\Gamma$ be finitely
generated and commensurated.  Suppose that, for some finite generating set
$S$, there is an abstract tree $U$ on
$V(\mathcal O_S(\Gamma,P))$ such that the identity maps in both directions
between $U$ and $\mathcal O_S(\Gamma,P)$ are Lipschitz.  Then there are a
finite symmetric generating set $S'\supseteq S$ and a spanning tree
$T\subseteq C_{S'}$ such that
\[
 |P\backslash\mathcal R_{S'}(T)|<\infty.
\]
\end{lemma}

\begin{proof}
The endpoints of a $U$-edge have uniformly bounded distance in the locally
finite orbital graph.  After translating the first endpoint to $P$, only
finitely many possibilities remain for the second endpoint.  Hence there is
a finite symmetric set $D\subseteq\Gamma$ such that, whenever $Px$ and $Py$
are joined in $U$, one may orient the edge so that
\[
 Pyx^{-1}=Pd\qquad\text{for some }d\in D.
\]
Enlarge $S$ by $D$ and by a finite symmetric generating set of $P$, and call
the result $S'$.

Fix a Cayley spanning tree of $P$.  In each fibre $Pg$ take its right
translate, which is a tree made of $S'$-edges.  For every edge
$\{Px,Py\}$ of $U$, choose $d\in D$ as above and add the Cayley edge
$\{x,dx\}$; its endpoints lie in $Px$ and $Py$.  Let $T$ be the union of
the fibre trees and these joining edges.  Contracting every fibre tree turns
$T$ into $U$.  Thus $T$ is connected and acyclic, hence is a spanning tree
of $C_{S'}$.

Because $S'$ is finite and the identity from the orbital graph to $U$ is
Lipschitz, the projections of the endpoints of every $S'$-edge have
$U$-distance at most one fixed number $k$.  After fibre contraction, its
$T$-replacement path projects to the $U$-geodesic between the two endpoint
fibres and therefore meets
at most $k+1$ fibres.  If $w$ is in an intermediate fibre and $q$ is the
terminal endpoint of the original edge, comparison of consecutive fibres
gives
\begin{equation}
 q w^{-1}\in P d_1P d_2P\cdots P d_mP,
 \qquad m\leq k,\quad d_i\in D.
 \label{eq:finite-double-coset-products}
\end{equation}
Indeed, for consecutive fibres $Px$ and $Py$ one has
$Pyx^{-1}=Pd$, hence $y=pd x$ for some $p\in P$; multiplying these
relations along the $U$-geodesic proves \eqref{eq:finite-double-coset-products}.
Since $P$ is commensurated, each $PdP$ and therefore each product in
\eqref{eq:finite-double-coset-products} is a finite union of right $P$-cosets.
There are only finitely many products with $m\leq k$ and $d_i\in D$, proving
the assertion.
\end{proof}

We shall also use the following relative form of the orbit-map lemma.

\begin{lemma}
\label{lem:orbit-map-comparison}
Suppose that a finitely generated group $\Gamma$ acts by graph automorphisms
and cocompactly on a connected locally finite graph $Z$, and let $H$ be a
vertex stabilizer.  If
$P\leq\Gamma$ is commensurable with $H$, then every orbital coset graph of
$(\Gamma,P)$ is quasi-isometric to $Z$.
\end{lemma}

\begin{proof}
For $P=H$, the map $Hg\mapsto g^{-1}o$, where $o$ is the chosen vertex, is
well defined.  Cocompactness and connectedness give $R\geq0$ such that every
vertex of $Z$ is at distance at most $R$ from $\Gamma o$.  Since $Z$ is
locally finite, $B_Z(o,2R+1)\cap\Gamma o$ is finite.  For every
$z\in B_Z(o,2R+1)\cap\Gamma o$, choose $a_z\in\Gamma$ such that
$a_z^{-1}o=z$.  Let $A\subseteq\Gamma$ be a finite symmetric set containing
all the elements $a_z$, and enlarge a finite generating set of $\Gamma$ by
$A$.  The map $Hg\mapsto g^{-1}o$ is Lipschitz for the resulting
orbital graph.

Conversely, replace each vertex of a geodesic in $Z$ by a point of
$\Gamma o$ at distance at most $R$, keeping orbit endpoints fixed.  Two
consecutive replacement points are at distance at most $2R+1$.  If they are
$g^{-1}o$ and $k^{-1}o$, set $z=(kg^{-1})^{-1}o$.  Then
$a_z\in A$ satisfies $a_z^{-1}o=z$.  Hence $kg^{-1}\in Ha_z$, so $Hg$ and $Hk$ are
adjacent in the enlarged orbital graph.  This proves that
$H\backslash\Gamma\to\Gamma o$ is a quasi-isometry, and $\Gamma o$ is
$R$-dense in $Z$.

For general $P$, put $L=P\cap H$.  Adjoin to the generating set finite sets
of representatives for $L\backslash P$ and $L\backslash H$.  Each of the
surjections
$L\backslash\Gamma\to P\backslash\Gamma$ and
$L\backslash\Gamma\to H\backslash\Gamma$ is then Lipschitz, and every fibre
has diameter at most one.  A path in either target orbital graph lifts edge
by edge to $L\backslash\Gamma$, with at most one final edge inside a fibre;
therefore both maps are quasi-isometries.  Changing back to any finite
generating set preserves the quasi-isometry classes.
\end{proof}

We now prove Theorem~\ref{thm:intro-main-commensurated}.  The proof also
shows that the Bass--Serre tree in condition~\textup{(iv)} is locally finite
and that all of its vertex and edge stabilizers are commensurable with $P$.

\begin{proof}[Proof of Theorem~\ref{thm:intro-main-commensurated}]
We first prove that \textup{(i)} implies \textup{(ii)}.  Let $S_0$ and a spanning tree
$T\subseteq C_{S_0}$ witness (i).  Enlarge $S_0$ to a finite symmetric
generating set $S$ containing a finite generating set of $P$.  The same
tree $T\subseteq C_S$ still witnesses $T$-POTP, since the tree-indexed
pseudo-orbit condition has not changed.  Put $E=\mathcal R_S(T)$.  Since $P$ is
commensurated, Corollary~\ref{cor:coset-full-shift-tree} gives
\begin{equation}
 |P\backslash E|<\infty.
\label{eq:commensurated-finite-paths}
\end{equation}
For $g\in\Gamma$, let $\mathcal T_{Pg}=\operatorname{Hull}_T(Pg)$ be the convex
hull of the fibre $Pg$ in $T$, that is, the union of all paths
$[u,v]_T$ with $u,v\in Pg$.  We claim that its image in
$\mathcal O_S(\Gamma,P)$ has uniformly bounded diameter.  If
$u,v\in Pg$, the path $[u,v]_T$ is contained in a finite union of
replacement paths of Cayley edges labelled by generators of a fixed finite generating set of
$P$.  Indeed, connect $u$ to $v$ inside the Cayley graph of the fibre using
those generators and replace every edge by its $T$-geodesic.  If $w$ lies
on such a replacement path with terminal endpoint $q$, then $r=qw^{-1}$ belongs to
$E$.  Choose a finite set $F\subseteq\Gamma$ with $E\subseteq PF$.  Writing
$r=pf$, we obtain $Pq=Prw=Pf w$, and hence
\begin{equation}
 d_{\mathcal O_S(\Gamma,P)}(Pw,Pq)\leq |f|_S.
\label{eq:replacement-orbital-radius}
\end{equation}
The derivation of \eqref{eq:replacement-orbital-radius} uses only that
$r=qw^{-1}$ belongs to $E$; hence it applies to every $S$-edge replacement
path, not only to those labelled by a generator of $P$.

Thus there is a number $D$, independent of $Pg$, such that
\begin{equation}
 \diam_{\mathcal O_S(\Gamma,P)}\pi(\mathcal T_{Pg})\leq D,
\label{eq:fibre-subtree-orbital-bound}
\end{equation}
where $\pi:\Gamma\to P\backslash\Gamma$ is the canonical projection onto
the right-coset space.

For a coset $C=Pg$, enlarge $\mathcal T_C$ to the connected subtree
\begin{equation}
 U_C=\mathcal T_C\cup
 \bigcup\{[h,sh]_T:s\in S,\ C\in\{Ph,Psh\}\}.
\label{eq:incident-replacement-subtree}
\end{equation}
Every added path meets $\mathcal T_C$ at its endpoint in $C$, and
\eqref{eq:replacement-orbital-radius} shows that $\pi(U_C)$ has diameter bounded
uniformly in $C$.  For a vertex $t$ of $T$, set
\begin{equation}
 \mathcal B_t=\{C\in P\backslash\Gamma:t\in U_C\}.
\label{eq:orbital-tree-bags}
\end{equation}
The set of $t$ for which a fixed coset $C$ belongs to $\mathcal B_t$ is
the connected subtree $U_C$.  If $Ph$ and $Psh$ are adjacent for some
$s\in S$, then
$[h,sh]_T\subseteq U_{Ph}\cap U_{Psh}$, so some bag contains both
endpoints.  These are precisely the two axioms for
$(T,(\mathcal B_t)_{t\in T})$ to be a tree decomposition.  Finally, if
$C\in\mathcal B_t$, then $\pi(t)$ lies in the uniformly bounded set
$\pi(U_C)$ containing $C$.  Hence every bag $\mathcal B_t$ lies in one
ball of uniformly bounded radius.  Since the orbital graph is locally
finite, the bags are finite as well, and the decomposition has uniformly
  bounded outer diameter.  Theorem~\ref{thm:background-Berger-Seymour}
makes that graph quasi-isometric to a tree.

For the implication from \textup{(ii)} to \textup{(i)},
Theorem~\ref{thm:background-Antolin} provides an abstract tree $U$ on the
vertex set of $\mathcal O_S(\Gamma,P)$ for which the identities in both
directions are Lipschitz.
Lemma~\ref{lem:uniform-coset-tree-lifting} gives $S'$ and $T$ with
$|P\backslash\mathcal R_{S'}(T)|<\infty$.  The commensurated case of
Corollary~\ref{cor:coset-full-shift-tree} proves (i).

We next prove that \textup{(ii)} holds if and only if \textup{(iii)} holds.
Put $G=\Gamma//P$.  The
orbital graph is a Cayley--Abels graph of the compactly generated group $G$,
so this equivalence is exactly
Theorem~\ref{thm:background-lc-quasitree}.

Assume (iii), subdividing edges if necessary to remove inversions, and let
$\mathcal T$ be the resulting locally finite tree.  Write
$\iota:\Gamma\to G$ for the dense homomorphism.  Vertex and edge stabilizers
in $G$ are compact open.  In the reduced Schlichting completion,
$\overline P$ is the stabilizer of the base coset, so
$\iota^{-1}(\overline P)=P$.  Any two compact open subgroups of $G$ are
commensurable; consequently the inverse image in $\Gamma$ of every vertex or
edge stabilizer is commensurable with $P$.

The dense subgroup $\iota(\Gamma)$ has the same vertex and edge orbits as
$G$: for a vertex or edge $z$, its stabilizer is open, so every coset
$gG_z$ meets $\iota(\Gamma)$.  Hence the quotient by $\Gamma$ is the same
finite graph as the quotient by $G$.
Theorem~\ref{thm:background-Bass-Serre} applied to the restricted action now
proves (iv).

Conversely, suppose (iv).  Commensurability is transitive, so each edge group
has finite index in each incident vertex group.  The Bass--Serre tree is
therefore locally finite, and the action is cocompact.  Every vertex
stabilizer is commensurable with $P$.
Lemma~\ref{lem:orbit-map-comparison} makes the orbital coset graph
quasi-isometric to this Bass--Serre tree, giving (ii).
\end{proof}

Since the fixed-tree criterion for an SFT depends only on $P_\ell(X)$,
Theorem~\ref{thm:intro-main-commensurated} also classifies every nonempty
SFT whose common left-period subgroup is finitely generated and
commensurated.

\section{Cohomological codimension one}
\label{sec:applications}

We now apply the commensurated classification to virtual cohomological
dimension.  For the coset shifts considered below, ordinary POTP follows
from finite generation of $P$, whereas Cayley-tree POTP is equivalent to
$\operatorname{vcd}(\Gamma)=\operatorname{vcd}(P)+1$.

Following \cite{Margolis}, a group $G$ is of type VFP over
$\mathbb Z$ if some finite-index subgroup $G_0$ admits a finite-length
resolution of the trivial $\mathbb ZG_0$-module $\mathbb Z$ by finitely
generated projective $\mathbb ZG_0$-modules.  Its virtual cohomological
dimension is $\operatorname{vcd}(G)=\operatorname{cd}_{\mathbb Z}(G_0)$;
this number is independent of the finite-index subgroup used to compute it.

The following two facts are Margolis's codimension-one theorem and
finite-index criterion.

\begin{theorem}
\cite[Theorem~1.2 and Proposition~1.3]{Margolis}
\label{thm:background-Margolis}
Let $H\leq G$ be a commensurated subgroup, and suppose that both $H$ and $G$
are of type VFP over $\mathbb Z$.
\begin{enumerate}[label=\textup{(\alph*)}]
\item If
$\operatorname{vcd}(G)=\operatorname{vcd}(H)+1$, then $G$ is the
fundamental group of a finite graph of groups in which every vertex and edge
group is commensurable with $H$.
\item The subgroup $H$ has finite index in $G$ if and only if
$\operatorname{vcd}(H)=\operatorname{vcd}(G)$.
\end{enumerate}
\end{theorem}

We shall also use Brown's cellular cohomological-dimension inequality.  We
state the general form followed by the tree specialization needed below.

\begin{theorem}
\cite[Chapter~VIII, \S2, Exercise~4, and \S3, Theorem~3.1]{Brown}
\label{thm:background-Brown-cellular-cd}
If a group $G$ acts cellularly on an acyclic $G$-complex $Y$, then
\[
 \operatorname{cd}_{\mathbb Z}(G)
 \leq
 \sup_{\sigma}
 \bigl\{\operatorname{cd}_{\mathbb Z}(G_\sigma)+\dim\sigma\bigr\},
\]
where $\sigma$ ranges over representatives of the $G$-orbits of cells and
$G_\sigma$ denotes the stabilizer of $\sigma$.  In particular, if $G$ acts
without inversions on a tree $\mathcal T$, then
\[
 \operatorname{cd}_{\mathbb Z}(G)
 \leq
 \max\left\{
   \sup_{v\in V\mathcal T}\operatorname{cd}_{\mathbb Z}(G_v),
   1+\sup_{e\in E\mathcal T}\operatorname{cd}_{\mathbb Z}(G_e)
 \right\},
\]
where $G_v$ and $G_e$ are the vertex and edge stabilizers.  An arbitrary
tree action reduces to this case by barycentric subdivision.
Moreover, if $G$ is torsion-free and $H\leq G$ has finite index, then
$\operatorname{cd}_{\mathbb Z}(H)=\operatorname{cd}_{\mathbb Z}(G)$.
\end{theorem}

We now prove the cohomological application stated in the introduction.

\begin{proof}[Proof of Corollary~\ref{cor:vcd-tree-POTP}]
Groups of type VFP are finitely generated.  Hence $\Gamma$ and $P$ are
finitely generated, and the relations
\[
 x(ph)=x(h)
 \qquad(p\in S_P,\ h\in\Gamma)
\]
for a finite generating set $S_P$ of $P$ define
$\mathcal A^{P\backslash\Gamma}$ as an
SFT.  It therefore has ordinary POTP by
Theorem~\ref{thm:CL-SFT}.

Suppose first that
$\operatorname{vcd}(\Gamma)=\operatorname{vcd}(P)+1$.  Part~(a) of
Theorem~\ref{thm:background-Margolis} gives a finite graph-of-groups
decomposition of $\Gamma$ in which every vertex and edge group is
commensurable with $P$.  The implication from \textup{(iv)} to \textup{(i)} in
Theorem~\ref{thm:intro-main-commensurated} gives Cayley-tree POTP.

Conversely, suppose that the coset shift has Cayley-tree POTP and put
$n=\operatorname{vcd}(P)$.  The same classification supplies a finite
graph-of-groups decomposition with every local group commensurable with $P$.
Let $\Gamma_0\leq\Gamma$ be a torsion-free finite-index subgroup of
cohomological dimension $\operatorname{vcd}(\Gamma)$, and let $\Gamma_0$ act
on the associated Bass--Serre tree.  Every vertex and edge stabilizer for
this action is a torsion-free group commensurable with $P$, hence has
cohomological dimension $n$ by the finite-index assertion in
Theorem~\ref{thm:background-Brown-cellular-cd}.  The dimension inequality for
a group acting on a tree, from the same theorem, now gives
\[
 \operatorname{vcd}(\Gamma)=\operatorname{cd}(\Gamma_0)\leq n+1.
\]
On the other hand,
$P\cap\Gamma_0\leq\Gamma_0$ has cohomological dimension $n$, so
$n\leq\operatorname{vcd}(\Gamma)$.  If equality held, then the
finite-index criterion in part~(b) of
Theorem~\ref{thm:background-Margolis} would give
$[\Gamma:P]<\infty$, contrary to the hypothesis.  Therefore
$\operatorname{vcd}(\Gamma)=n+1$.
\end{proof}

Thus ordinary POTP holds whenever the period subgroup $P$ is finitely
generated, while Cayley-tree POTP is equivalent, under the hypotheses of
Corollary~\ref{cor:vcd-tree-POTP}, to
$\operatorname{vcd}(\Gamma)-\operatorname{vcd}(P)=1$.

\begin{example}
\label{ex:single-splitting-not-enough}
Let $P=\Z$ and $\Gamma=P\times(\Z^2*\Z^2)$, with $P$ embedded as the first
factor.  Then $P$ is finitely generated, normal, and of infinite index, and
\[
 \Gamma=(P\times\Z^2)*_P(P\times\Z^2).
\]
Thus $\Gamma$ already has a nontrivial splitting over $P$.  Nevertheless,
the orbital coset graph is quasi-isometric to $\Z^2*\Z^2$, which is not a
quasi-tree: a finitely generated group quasi-isometric to a tree is virtually
free, whereas $\Z^2*\Z^2$ contains $\Z^2$.  Consequently
$\mathcal A^{P\backslash\Gamma}$ has ordinary POTP but has no Cayley-tree POTP by
Theorem~\ref{thm:intro-main-commensurated}.  Hence the existence of one
splitting over $P$ does not imply any of the equivalent conditions
\textup{(i)--(iv)} in that theorem.
\end{example}

\section{Finite-index extensions and generic Cayley-tree POTP}
\label{sec:generic-STRP}

This section combines Theorem~\ref{thm:intro-main-universal} with generic
Cantor dynamics.  The finite-index argument below also answers the virtually cyclic
STRP problem posed by Doucha.

Let $G$ be countable and let $\mathcal A$ be a finite alphabet.  Denote by
$\mathcal S_G(\mathcal A)$ the compact space of nonempty $G$-subshifts of
$\mathcal A^G$ with the Vietoris topology.  If $F\subseteq G$ is finite and
$X\in\mathcal S_G(\mathcal A)$, put
\[
 X_F=\{x|_F:x\in X\},
 \qquad
 \mathcal N_X^F
 =\{Y\in\mathcal S_G(\mathcal A):Y_F=X_F\}.
\]
These sets form a clopen neighborhood basis.  Following
\cite[Definition~2.22]{DouchaSTRP}, a subshift
$X\in\mathcal S_G(\mathcal A)$ is \emph{projectively isolated} if there are
a finite alphabet $\mathcal B$, a subshift
$Y\in\mathcal S_G(\mathcal B)$, a block map
$\phi:\mathcal B^G\to\mathcal A^G$, and an open neighborhood
$\mathcal U$ of $Y$ in $\mathcal S_G(\mathcal B)$ such that
$\phi(Z)=X$ for every $Z\in\mathcal U$.  Since the sets
$\mathcal N_Y^E$ form a neighborhood basis, one may choose a finite
$E\subseteq G$ such that $\phi(Z)=X$ for every
$Z\in\mathcal N_Y^E$.

We use the following results in precisely the forms stated here.

\begin{theorem}
\cite[Theorems~3.1, 4.12, 6.3 and~6.9]{DouchaSTRP}
\label{thm:background-Doucha-STRP}
Let $G$ be a countable group.
\begin{enumerate}[label=\textup{(\alph*)}]
\item The group $G$ has STRP if and only if, for every finite alphabet
$\mathcal A$ with $|\mathcal A|\geq2$, the projectively isolated subshifts
are dense in $\mathcal S_G(\mathcal A)$.
\item Every finitely generated free group has STRP.
\item If $G$ is finitely generated, then ordinary POTP is generic in
$\operatorname{Act}_G(\mathfrak C)$ if and only if $G$ has STRP.
\end{enumerate}
\end{theorem}

\begin{lemma}
\cite[Lemmas~2.25 and~2.27]{DouchaSTRP}
\label{lem:background-projective-isolation}
A factor of a projectively isolated subshift is projectively isolated.
Moreover, a projective-isolation witness may be recoded so that its
isolating factor map is induced by a one-block alphabet map.
\end{lemma}

Doucha uses the left shift.  The coordinate inversion
$Jx(k)=x(k^{-1})$ conjugates it to our right-shift convention, so
Theorem~\ref{thm:background-Doucha-STRP} and
Lemma~\ref{lem:background-projective-isolation} apply to the right shift.

\begin{theorem}
\label{thm:finite-index-STRP}
Let $G$ be finitely generated and let $H\leq G$ have finite index.  If $H$
has STRP, then $G$ has STRP.
\end{theorem}

\begin{proof}
Put $Q=G/H$, regarded as the finite transitive left $G$-set of left cosets,
and let $q_0=H$.  Choose representatives
$R=\{r_q:q\in Q\}$ with $r_{q_0}=e$.  Thus every element of $G$ has a
unique expression $r_qh$, with $q\in Q$ and $h\in H$.

We first construct, from an $H$-subshift, a $G$-subshift with an auxiliary
$Q$-coordinate recording the left coset, and prove that this construction
preserves projective isolation.
For a finite alphabet $\mathcal C$, define
\[
 j_{\mathcal C}:\mathcal C^H\longrightarrow\mathcal C^G,
 \qquad
 j_{\mathcal C}(z)(r_qh)=z(h),
\]
and, for $q\in Q$, define $\omega_q\in Q^G$ by
\[
 \omega_q(k)=kq.
\]
Thus the value of the auxiliary coordinate at $k$ is the coset $kq$, and
$\sigma_g\omega_q=\omega_{gq}$.  If
$Z\subseteq\mathcal C^H$ is an $H$-subshift, put
\begin{equation}
 I_{\mathcal C}(Z)
 =\bigcup_{q\in Q}
 \{(\sigma_{r_q}j_{\mathcal C}(z),\omega_q):z\in Z\}
 \subseteq(\mathcal C\times Q)^G.
 \label{eq:coset-label-induction}
\end{equation}
For $g\in G$ write $gr_q=r_{gq}h$ with $h\in H$.  Since
$j_{\mathcal C}(\sigma_hz)=\sigma_hj_{\mathcal C}(z)$ and
$\sigma_g\sigma_{r_q}=\sigma_{gr_q}$, the set
$I_{\mathcal C}(Z)$ is a $G$-subshift.

Suppose that $Y\subseteq\mathcal B^H$ is projectively isolated.  By
Lemma~\ref{lem:background-projective-isolation}, there are an
$H$-subshift $Z\subseteq\mathcal C^H$, a one-block map
$\phi_0:\mathcal C\to\mathcal B$, and a finite set $E\subseteq H$ with
$e\in E$ such that
\begin{equation}
 \phi(Z')=Y
 \qquad\text{for every }Z'\in\mathcal N_Z^E,
 \label{eq:H-projective-witness}
\end{equation}
where $\phi$ is induced by $\phi_0$.  Define the map
$\Phi:(\mathcal C\times Q)^G\to(\mathcal B\times Q)^G$ by
\nopagebreak[4]
\begin{equation}
 \Phi(x,\theta)(k)
 =\bigl(\phi_0(x(r_{\theta(k)}^{-1}k)),\theta(k)\bigr).
 \label{eq:coset-label-factor}
\end{equation}
This is a block map with memory contained in $R^{-1}\cup\{e\}$.  It is
$G$-equivariant: both
$\Phi(\sigma_g(x,\theta))(k)$ and
$\sigma_g\Phi(x,\theta)(k)$ read
\[
 x(r_{\theta(kg)}^{-1}kg)
 \quad\text{and}\quad \theta(kg).
\]
If $k=r_th$, then $\omega_{q_0}(k)=t$ and
$r_t^{-1}k=h$.  Hence
\begin{equation}
 \Phi(j_{\mathcal C}(z),\omega_{q_0})
 = (j_{\mathcal B}(\phi(z)),\omega_{q_0}),
 \label{eq:base-label-factor}
\end{equation}
and equivariance gives
$\Phi(I_{\mathcal C}(Z))=I_{\mathcal B}(Y)$.

It remains to find a finite set $M\subseteq G$ such that the same map
$\Phi$ sends every subshift $L$ satisfying
$L_M=I_{\mathcal C}(Z)_M$ onto $I_{\mathcal B}(Y)$.
Choose a finite symmetric generating set $S$ of $G$ and put
$K_0=\{e\}\cup S$.  The $K_0$-language isolates the finite orbit of
coset-label configurations
\[
 \Omega=\{\omega_q:q\in Q\}.
\]
Indeed, if a $Q$-configuration has only $\Omega$-allowed
$K_0$-patterns, then translating the window to $u\in G$ gives
\[
 \theta(su)=s\theta(u)
 \qquad(s\in S),
\]
and therefore $\theta(k)=k\theta(e)$ for every $k\in G$.

Take a finite $M\supseteq E\cup K_0$, and let
$L\subseteq(\mathcal C\times Q)^G$ be a subshift satisfying
\[
 L_M=I_{\mathcal C}(Z)_M.
\]
Its projection to $Q^G$ is a nonempty $G$-invariant subset of the transitive
orbit $\Omega$, and hence is all of $\Omega$.  The fibre over
$\omega_{q_0}$,
\[
 L_{q_0}=\{x\in\mathcal C^G:(x,\omega_{q_0})\in L\}
\]
is nonempty, compact, and $H$-invariant.  Matching joint $M$-patterns in
both directions, the value of the $Q$-coordinate at $e$ forces the matching
component to have coset-label configuration $\omega_{q_0}$.  Consequently
\[
 Z_L:=\{x|_H:x\in L_{q_0}\}
 \quad\text{satisfies}\quad
 (Z_L)_E=Z_E.
\]
The set $Z_L$ is a nonempty $H$-subshift.  Thus
$Z_L\in\mathcal N_Z^E$ and
$\phi(Z_L)=Y$ by \eqref{eq:H-projective-witness}.  Formula
\eqref{eq:coset-label-factor} gives, for every $x\in\mathcal C^G$,
\[
 \Phi(x,\omega_{q_0})
 =\bigl(j_{\mathcal B}(\phi(x|_H)),\omega_{q_0}\bigr).
\]
It follows that
\[
 \Phi(L_{q_0}\times\{\omega_{q_0}\})
 =j_{\mathcal B}(Y)\times\{\omega_{q_0}\}.
\]
For each $q\in Q$, the fibre of $L$ over $\omega_q$ is the
$\sigma_{r_q}$-translate of the fibre over $\omega_{q_0}$.  Equivariance therefore
gives
\[
 \Phi(L)=I_{\mathcal B}(Y).
\]
We have proved that $I_{\mathcal B}(Y)$ is projectively isolated.

We now prove that projectively isolated $G$-subshifts are dense.  Fix a
finite alphabet $\mathcal A$, a
$G$-subshift $X\subseteq\mathcal A^G$, and a finite set $F\subseteq G$.
Put $\mathcal B=\mathcal A^R$ and define the $H$-equivariant homeomorphism
\[
 \Psi:\mathcal A^G\longrightarrow\mathcal B^H,
 \qquad
 \Psi(x)(h)(r)=x(rh).
\]
Let $\widehat X=\Psi(X)$ and
\[
 K=\bigcup_{q\in Q}Fr_q.
\]
Choose a finite $D\subseteq H$ such that
$K\subseteq RD:=\{rh:r\in R,\ h\in D\}$.  Since $H$ has STRP,
Theorem~\ref{thm:background-Doucha-STRP}\textup{(a)} gives a projectively
isolated $H$-subshift $Y\subseteq\mathcal B^H$ with
$Y_D=\widehat X_D$.  Put $V=\Psi^{-1}(Y)$.  Then
$V_{RD}=X_{RD}$, and in particular $V_K=X_K$.

Set
\[
 U=\bigcup_{q\in Q}\sigma_{r_q}V.
\]
If $g r_q=r_{gq}h$ with $h\in H$, then
\[
 \sigma_g(\sigma_{r_q}V)
 =\sigma_{r_{gq}}(\sigma_hV)
 =\sigma_{r_{gq}}V.
\]
Thus $U$ is a $G$-subshift.  Moreover,
\[
 (\sigma_{r_q}V)_F=(\sigma_{r_q}X)_F=X_F
 \qquad(q\in Q),
\]
because the left-hand language is determined by $V_{Fr_q}$ and
$Fr_q\subseteq K$.  Hence $U_F=X_F$.

Finally define
\begin{equation}
 \Theta:(\mathcal B\times Q)^G\longrightarrow\mathcal A^G,
 \qquad
 \Theta(x,\theta)(k)
 =x(r_{\theta(k)}^{-1}k)(r_{\theta(k)}).
 \label{eq:coset-label-forgetting-factor}
\end{equation}
The map $\Theta$ has memory contained in $R^{-1}\cup\{e\}$, and
\[
 \Theta(\sigma_g(x,\theta))(k)
 =x(r_{\theta(kg)}^{-1}kg)(r_{\theta(kg)})
 =(\sigma_g\Theta(x,\theta))(k),
\]
so it is a $G$-equivariant block map.  If $k=r_th$, then
\[
 \Theta(j_{\mathcal B}(y),\omega_{q_0})(r_th)=y(h)(r_t),
\]
so $\Theta(j_{\mathcal B}(y),\omega_{q_0})=\Psi^{-1}(y)$.  Therefore
\[
 \Theta(I_{\mathcal B}(Y))=U.
\]
The induced subshift $I_{\mathcal B}(Y)$ is projectively isolated, and
$U$ is projectively isolated by
Lemma~\ref{lem:background-projective-isolation}.  Since $X$ and $F$ were
arbitrary, projectively isolated $G$-subshifts are dense for every finite
alphabet.  Theorem~\ref{thm:background-Doucha-STRP}\textup{(a)} now gives
STRP for $G$.
\end{proof}

\begin{proof}[Proof of Corollary~\ref{cor:generic-tree-virtually-free}]
Let $\Gamma$ be finitely generated and virtually free.  It has a
finite-index finitely generated free subgroup $H$.  The trivial group has
STRP, and every nontrivial finitely generated free group has STRP by
Theorem~\ref{thm:background-Doucha-STRP}\textup{(b)}.  Thus
Theorem~\ref{thm:finite-index-STRP} gives STRP for $\Gamma$.

By Theorem~\ref{thm:background-Doucha-STRP}\textup{(c)}, ordinary POTP is
comeager in $\operatorname{Act}_\Gamma(\mathfrak C)$.  For a virtually free
group, Theorem~\ref{thm:intro-main-universal} implies that every such action
has Cayley-tree POTP, while Cayley-tree POTP always implies ordinary POTP.
Hence the two classes of actions coincide and Cayley-tree POTP is
comeager.  Every virtually cyclic group is finitely generated and virtually
free, giving the final assertion.
\end{proof}

Theorem~\ref{thm:finite-index-STRP} establishes the finite-index-overgroup
direction of Doucha's permanence question and answers the virtually cyclic
problem.  The finite-index-subgroup direction, and hence full
commensurability invariance, remains open.

\end{document}